\documentclass[reqno, a4paper]{amsart}
\usepackage[centering]{geometry}
\usepackage{amsmath,amssymb,amsthm}
\usepackage[scr=boondox]{mathalfa}
\usepackage{xcolor}
\usepackage[colorlinks=true, linkcolor=black, citecolor=black, urlcolor=black]{hyperref}

\usepackage{enumerate} % [(i)] kann man machen oki
\usepackage{cleveref}

\newtheorem{theorem}{Theorem}
\newtheorem{assumption}[theorem]{Assumption}
\newtheorem{corollary}[theorem]{Corollary}
\newtheorem{definition}[theorem]{Definition}
\newtheorem{lemma}[theorem]{Lemma}
\newtheorem{proposition}[theorem]{Proposition}

\theoremstyle{remark}
\newtheorem{remark}[theorem]{Remark}
\newtheorem{example}[theorem]{Example}

\numberwithin{equation}{section}
\numberwithin{theorem}{section}

\newcommand{\MBtext}[1]{{{\small\color{purple}{}}}}

\newcommand{\LRlong}[1]{{{\small\color{cyan}{}}}}

\newcommand{\LRtext}[1]{{{\small\color{cyan}{}}}}

\newcommand{\STtext}[1]{{{\small\color{green}{}}}}

\newcommand{\cL}{\mathcal{L}}

\newcommand{\cP}{\mathcal{P}}
\newcommand{\Pc}{\cP}
\renewcommand{\P}{\cP}

\newcommand{\law}{\text{law}}
\newcommand{\EE}{\mathbb E}
\newcommand{\F}{\mathcal F}

\newcommand{\cW}{\mathcal{W}}

\newcommand{\Leb}{\text{Leb}}

\newcommand{\R}{\mathbb{R}}
\newcommand{\N}{\mathbb{N}}

\newcommand{\cpl}{\mathrm{Cpl}}
\newcommand{\cplm}{\cpl_\mathrm{M}}

\newcommand{\dom}{\mathrm{dom}}
\newcommand{\OT}[3]{\mathrm{T}_#1(#2,#3)}
\newcommand{\WOT}[3]{\mathrm{WT}_#1(#2,#3)}

\newcommand{\EOT}[4]{\mathrm{ET}_{#1,#2}(#3,#4)}
\newcommand{\DW}[3]{\mathrm{D}_#1(#2,#3)}
\newcommand{\Dadm}[3]{\Phi_#1(#2,#3)}

\newcommand{\eps}{\varepsilon}
\renewcommand{\epsilon}{\varepsilon}
\newcommand{\mean}[1]{\mathrm{mean}(#1)}

\newcommand{\spt}{\mathrm{spt}}

\newcommand{\W}{\mathcal{W}}
\newcommand{\MCov}{\mathrm{MCov}}
\newcommand{\id}{\textup{id}}

\newcommand{\pphi}{f}
\newcommand{\ppsi}{g}
\title{The fundamental theorem of weak optimal transport}

\author{M.\ Beiglböck, G.\ Pammer, L.\ Riess, S.\ Schrott}

\date{\today}

\begin{document}

\begin{abstract}
The fundamental theorem of classical optimal transport establishes strong duality and characterizes optimizers through a complementary slackness condition. Milestones such as Brenier's theorem and the Kantorovich-Rubinstein formula are direct consequences. 

In this paper, we generalize this result to non-linear cost functions, thereby establishing a fundamental theorem for the weak optimal transport problem introduced by Gozlan, Roberto, Samson, and Tetali. As applications we provide concise derivations of the Brenier--Strassen theorem, the convex Kantorovich--Rubinstein formula and the structure theorem of entropic optimal transport. 
We also extend Strassen's theorem in the direction of Gangbo--McCann's transport problem for convex costs. 
Moreover, we determine the optimizers for a new family of transport problems which contains the Brenier--Strassen, the martingale Benamou--Brenier and the entropic martingale transport problem as extreme cases.
\medskip

\emph{keywords:} weak optimal transport, Gangbo-McCann theorem, dual attainment, martingale optimal transport, entropic optimal transport, adapted weak topology
\end{abstract}
\maketitle

\section{Introduction}\label{sec:intro}

Optimal transport (OT) provides a powerful theory with applications in numerous fields. However, there is a range of problems that cannot be addressed within this framework due to non-linearities in the cost function. This motivated Gozlan, Roberto, Samson and Tentali \cite{GoRoSaTe14} to introduce the more encompassing framework of \emph{weak optimal transport} (WOT), see also the 
closely related works of Aliberti, Bouchitte and Champion \cite{AlBoCh18} as well as Bowles and Ghoussoub \cite{BoGh18}.
Our goal is to extend the basic structure theorem of transport, the \emph{fundamental theorem of optimal transport} (see e.g.\ \cite[Theorem 1.13]{AmGi13}), to the WOT framework. To set the stage for this result and its applications, we briefly recount the classical setup.

Given a Polish metric space $Z$, we denote by $\cP(Z)$ the set of probabilities on $Z$, equipped with weak convergence and by $\cP_p(Z)$, $p\in [1,\infty)$ probabilities with finite $p$-moment with $p$-weak convergence. 
Throughout we consider Polish probability spaces $(X, \mu)$, $(Y,\nu)$. We write $\cpl(\mu, \nu)$ for the set of all \emph{couplings} or \emph{transport plans}, that is, probabilities on \(X \times Y\) which have \(X\)-marginal \(\mu\) and \(Y\)-marginal \(\nu\). As customary, we assume that the `\emph{cost function}' $c:X\times Y\to [0,\infty]$ is lsc, implying attainment for the (primal) \emph{Monge-Kantorovich transport problem}
\begin{align}\label{eq:PrimalOT}
   {\rm T}_c(\mu, \nu):= \inf_{\pi \in \cpl(\mu, \nu)}\int c(x,y)\, \pi(dx, dy). 
\end{align}
In key applications we are often in the situation that
$X=Y$ and $c$ relates to the metric $d_X$ of $X$: 
E.g.\ \emph{Kantorovich-Rubinstein} \cite{KaRu58} asserts that for $\mu, \nu \in \cP_1(X)$ and $c=d_X$ we have ${\rm T}_{d_X}(\mu, \nu)= \max_{\text{$\psi$ $1$-Lip}} \nu(\psi)-\mu(\psi).$

Likewise \emph{Brenier's theorem} \cite{Br91} asserts that for $X=Y=\R^d$, the squared distance cost $c(x,y)= |x-y|^2$ and $\mu, \nu\in \cP_2(\mu, \nu)$, $\mu\ll \Leb$ the transport problem admits a unique solution. Moreover, the optimizer is characterized as the only admissible transport plan supported by the graph of a convex function. 

Both results are consequences of the basic structure theorem of optimal transport: 

\begin{theorem}[Fundamental Theorem of OT]\label{thm:FTOT}
For $c:X\times Y\to [0,\infty]$ lsc, the primal problem ${\rm T}_c(\mu, \nu)$ is attained and equals the dual problem
\begin{align}
    {\rm D}_c(\mu, \nu):= \sup\{\mu(\pphi)+\nu(\ppsi): \pphi(x) + \ppsi(y) \leq c(x,y), \, \pphi\in \mathcal L^1(\mu), \, \ppsi \in \mathcal L^1(\nu)\}. 
\end{align}
The dual problem is attained if $c$ is upper bounded in the sense that 
\begin{align}\label{eq:OTbound} \tag{{\sf A}}
    c(x,y)\leq a(x) + b(y) \quad \text{ for some } a\in \mathcal L^1(\mu), \, b \in \mathcal L^1(\nu).
\end{align}
Candidates $\pi, (\pphi, \ppsi)$ are optimal if and only if they satisfy the complementary slackness condition 
\begin{align}
    c(x,y)= \pphi(x)+\ppsi(y), \quad \pi\text{-a.s.}
\end{align}    
\end{theorem}
Given a measurable \emph{cost function} $C:X\times \P_p(Y)\to [0,\infty]$ which is lsc and convex in the second argument, and denoting the $\mu$-disintegration of $\pi$ by $(\pi_x)_{x\in X}$,
the weak transport problem is 
\begin{align}\label{eq:primalWOT}
  {\rm WT}_C(\mu, \nu)=\inf_{\pi \in \cpl(\mu, \nu)}\int C(x,\pi_x)\, \mu(dx).   
\end{align}
The classical transport problem corresponds to the case where $\rho\mapsto C(x,\rho)$ is even linear, i.e.
\begin{align} \label{eq:linearC}
    C(x, \rho)= \int c(x,y) \, \rho(dy).
\end{align}
To establish the desired counterpart of Theorem \ref{thm:FTOT} we need the \emph{boundedness condition} 
\begin{align}\label{eq:WOTbound} \tag{{\sf{B}}}
    C(x,\rho) \le a(x) + \rho (b)  + \int h\Big( \frac{ d\rho}{d\nu}\Big)d\nu,  
\end{align} for some $a \in \mathcal L^1(\mu), b \in \mathcal L^1(\nu)$ and a convex increasing function $h : [0,\infty)\to [0,\infty)$, 
    where the right-hand side of \eqref{eq:WOTbound} is understood as $+\infty$ if $\rho$ is not absolutely continuous w.r.t.\ $\nu$. Apparently \eqref{eq:OTbound} implies \eqref{eq:WOTbound} if $C$ is of transport type as in \eqref{eq:linearC}. 
    Some boundedness assumption is necessary to obtain dual attainment in optimal transport, see e.g.\ \cite[Examples 4.4, 4.5]{BeSc11}, and clearly these examples also pertain to weak transport.

We also need that for increasing sequences $(Y_k)_k$ of Borel sets with $\bigcup_k Y_k=Y$ 
\begin{align}\label{eq:WOTcont} \tag{{\sf{C}}}
    C(x, \rho) \ge \limsup_k C(x, (\rho|_{Y_k}) / \rho(Y_k)).
\end{align}
This is as a mild \emph{continuity condition}, trivially satisfied in the classical setting \eqref{eq:linearC} and our intended applications. 
However, see Example \ref{exmpl:asmp_WOTcont} for the necessity of some continuity condition. 

\medskip

With these preparations, we can now state our main theorem: 
\begin{theorem}[Fundamental Theorem of WOT]\label{thm:FTWOTintro}
If $C:X\times \P_p(Y)\to [0,\infty]$ is measurable with $\rho\mapsto C(x,\rho)$ being convex and lsc for the $p$-weak convergence on $\P_p(Y)$, then $\WOT{C}{\mu}{\nu}$ is attained and equals the dual problem\footnote{Specifically we require the inequality in \eqref{eq:WOTdual_intro} for all $(x,\rho) \in X \times \P_p(Y)$ with $\ppsi \in \mathcal L^1(\rho)$.}
\begin{align}\label{eq:WOTdual_intro}
    {\rm D}_C(\mu, \nu):= \sup\{\mu(\pphi)+\nu(\ppsi): \pphi(x) + \rho(\ppsi) \leq C(x,\rho), \, \pphi\in \mathcal L^1(\mu), \, \ppsi \in \mathcal L^1(\nu)\}. 
\end{align}
Moreover, if $C$ satisfies conditions \eqref{eq:WOTbound} and \eqref{eq:WOTcont},
then the dual problem is attained. A pair of candidates $(\pi, (\pphi, \ppsi))$ is optimal if and only if it satisfies the complementary slackness condition 
\begin{align}\label{eq:WOTCS1}
    C(x,\pi_x)= \pphi(x)+\pi_x(\ppsi), \quad \mu\text{-a.s.}
\end{align}    
\end{theorem}

Note that analogous to the classical transport problem we can rewrite \eqref{eq:WOTdual_intro} as
\begin{align*} {\rm D}_C(\mu, \nu)= & \sup\{\mu(g^C) +\nu(g): g^C \in \mathcal L^1(\mu), \, g\in \mathcal L^1(\nu)\}
\\ 
\text{where}\      g^C (x) := &\inf \{ C(x,\rho) - \rho(g) : \rho \in \cP_p(Y) \text{ s.t.\ } g \in \cL^1(\rho) \}.
\end{align*}
In fact, the core of utilizing \Cref{thm:FTWOTintro} is the computation of the $C$-transform $g^C$ for chosen instances of $C$.

\subsection{Applications of the Fundamental Theorem of WOT}
Next we give an overview of our applications of Theorem~\ref{thm:FTWOTintro} which consist in shortened, unified derivations of old results as well as of new results following the spirit of Kantorovich-Rubinstein / Brenier's theorem.

\subsubsection{Convex Kantorovich--Rubinstein}
Given $\mu, \nu\in \P_1(\R^d)$, \cite{GoRoSaTe15, AlBoCh18} derive the following \emph{convex Kantorovich--Rubinstein formula} 
\begin{align}\label{eq:coKaRu}
  \inf_{\pi \in \cpl(\mu, \nu)}\int |x-\mean{\pi_x}|\, \mu(dx) = \max_{\text{$\psi$ convex, $1$-Lip}} \mu(\psi)-\nu(\psi). 
\end{align}
A transport plan $\pi$ is called a martingale coupling if it is the law of a two-step martingale, that is, if $\mu$-a.s.\ $\mean{\pi_x}=x$. A classical result of Strassen \cite{St65} asserts that the set $\cplm(\mu, \nu)$ of martingale couplings is non-empty if and only if $\mu, \nu$ are in convex order $\leq_c$, i.e.\ $\mu(\psi)\leq \nu(\psi)$ for all convex $\psi$. 
The convex Kantorovich-Rubinstein formula can be seen as quantitative extension of Strassen's theorem: the left hand side of \eqref{eq:coKaRu} measures how close we are to finding a martingale coupling of $\mu$ and $\nu$, while the right hand side measures by how much $\mu, \nu$ violate the convex order condition.

\subsubsection{The Brenier--Strassen theorem}
Given $\mu, \nu\in \Pc_2(\R^d)$,
Gozlan--Julliet \cite{GoJu18} establish a Brenier--Strassen theorem for the weak transport problem 
\begin{align}\label{eq:coBrSt}
  \inf_{\pi \in \cpl(\mu, \nu)}\int |x-\mean{\pi_x}|^2\, \mu(dx). 
\end{align}
It asserts that there exists a $\mu$-a.s.\ unique map $T:\R^d\to \R^d$ such that for every optimizer $\pi \in \cpl(\mu, \nu)$, $T(x)= \mean{\pi_x}$. Moreover $T$ is $1$-Lipschitz, the gradient of a convex function and optimizers of \eqref{eq:coBrSt} are characterized by
\begin{align}\label{eq:BrStoptimizers}
\pi\in \cpl(\mu, \nu) \textrm{ is optimal } \Longleftrightarrow \   \textrm{ $\pi_x= \kappa_{T(x)}$, $\mu$-a.s.\ for some $\kappa\in \cplm(T_\#\mu, \nu).$}
\end{align}
We emphasize that in contrast to Brenier's theorem, no regularity of the measure $\mu$ is required beyond the existence of finite second moment. 
In \cite{GoJu18}, the Brenier--Strassen theorem is 
established based on duality for compactly supported probabilities and using approximation arguments to pass to the general case, see also \cite{BaBePa19}.

\subsubsection{A Gangbo--McCann--Strassen theorem}
In Theorem \ref{thm:StGaMc} below, we establish a weak transport version of Gangbo--McCann's result \cite{GaMc96} which recovers the convex Kantorovich--Rubinstein formula as well as the Brenier--Strassen theorem as special cases. Specifically we consider for convex $\vartheta : \R^d\to [0,\infty)$ the problem
\begin{align}
    \label{eq:coGaMcSt}
  \inf_{\pi \in \cpl(\mu, \nu)}\int \vartheta (x-\mean{\pi_x})\, \mu(dx). 
\end{align}
We show that if $\vartheta $ is strictly convex, there exists a unique $T$ such that for $\pi$ optimal $T(x)= \mean{\pi_x}$, $\mu$-a.s. Moreover the set of optimizers is characterized as in \eqref{eq:BrStoptimizers}. 

On the other hand, if $\vartheta $ is positively homogeneous, (i.e.\ a support function), we obtain an extended convex Kantorovich--Rubinstein formula. Beyond \eqref{eq:coKaRu} it implies for instance an increasing convex Kantorovich--Rubinstein formula, giving a quantitative version of Strassen's result for sub-martingales / the increasing convex order (see Corollary \ref{cor:hSpt}).  

We note that barycentric costs can be seen as a relaxation of the martingale condition which appears naturally from mathematical finance considerations.  Here the martingale condition for transport plans corresponds to a strict No-Arbitrage paradigm, this is a robust version of the `fundamental theorem of asset pricing', e.g.\ \cite{HoNe12, AcBePeSc13, WiZh22}. Based on \eqref{eq:coKaRu} and \eqref{eq:coGaMcSt} we obtain quantitative extensions: The level of arbitrage that can be locked in under a given trading restriction is reflected by a corresponding relaxation of the martingale condition, see Section \ref{ssec:Finance_interpretation}.

\subsubsection{Entropic optimal transport}

Given probabilities $\mu, \nu$ on Polish spaces $X,Y$, a measurable cost function $c:X\times Y\to [0,\infty)$ which is bounded by the sum of two integrable functions, and $\eps >0$ the entropic transport problem reads
\begin{align}\label{eq:EOTWOT_intro}
  \inf_{\pi \in \cpl(\mu,\nu)} \int c\, d\pi + \eps H ( \pi | \mu \otimes \nu),  
\end{align}
where $H$ denotes relative entropy, i.e.\
$H(Q | R ) = \int \log\big( \frac{dQ}{dR} \big) \,d Q$ if $Q, R$ are probability measures with $Q\ll R$ and $H(Q | R ) =\infty $ otherwise. See \cite{Le13, Nu22} for surveys on entropic optimal transport (EOT) and the closely related Schrödinger problem.

Observing that $H ( \pi | \mu \otimes \nu) = \int H(\pi_x|\nu)\,  \mu(dx)$, we notice that \eqref{eq:EOTWOT_intro} is a weak transport problem.   
Based on Theorem \ref{thm:FTWOTintro} we recover the main structure theorem of EOT: \eqref{eq:EOTWOT_intro} admits a unique solution $\pi$ and among all couplings of $\mu$ and $\nu$ the optimizer is characterized by   
\begin{align}\label{eq:proddensity_intro}
    \frac{d\pi}{d\mu \otimes \nu}(x,y) = 
    \exp\left(- \frac{c(x,y) + \pphi(x) + \ppsi(y) }{\eps}\right).
\end{align}
 In this case $(\pphi,\ppsi)$ is optimal in the dual problem \eqref{eq:EOTWOTdual}.

Along the same lines we also prove the structure theorem for transport costs regularized by general convex functions, thus recovering a result of \cite{DiGe20}, see Section \ref{sec:RWGCF}.

\subsubsection{Relaxation of martingale transport}
Weak martingale optimal transport (WMOT) differs from classical transport in that one minimizes over martingale couplings $\cplm(\mu, \nu)$ for $\mu,\nu\in\cP_p(\R^d)$. This restriction can be incorporated by considering weak transport costs of the form 
\begin{align*}%\label{eq:WMOTcostintro}
    C (x,\rho) =\begin{cases}
        \widehat C(\rho) & \mean{\rho} = x\\
        \infty & \text{else.}
    \end{cases}
\end{align*}
Here $C $ is lsc and convex in the second argument provided that $\widehat C:\P_p(\R^d)\to [0,\infty]$ is lsc and convex on all fibers $\{\rho\in \P_p(\R^d): \mean{\rho}=x\}, x\in \R^d$. However, $C $ 
does not satisfy the boundedness assumption \eqref{eq:WOTbound}. In fact dual attainment for martingale transport fails even in very regular settings (e.g.\ \cite[Section 4.3]{BeHePe12}). Positive results are only available for $d=1$ and under strong assumptions (\cite{BeLiOb19, BeNuTo16}). In consequence, important questions such as the entropic martingale transport and the martingale Benamou--Brenier problem are only partially understood. 

In \Cref{sec:RWMT} we show that these challenges can be overcome by applying \Cref{thm:FTWOTintro} to a suitable relaxation. In particular, if $\widehat C$ satisfies appropriate counterparts to \eqref{eq:WOTbound} and \eqref{eq:WOTcont}, we establish dual attainment for WMOT, 
which is relaxed using a term $\beta \W_2^2$, where $\beta >0$: 
\begin{align}\label{eq:RWMOTintro}
   \min_{\eta\in \P_2(\R^d)} \Big( \beta \W^2_2(\mu, \eta)  +  \min_{\kappa\in \cplm(\eta, \nu)} \int C (x,\pi_x)\, \mu(dx) \Big) = \max_{g\in \mathcal L^1(\nu)} \mu\big(%\beta^{-1} 
  \tfrac1\beta |.|^2 \Box g^{C }\big) + \nu\big(g\big).
\end{align}

We note that while the dual uses the classical infimal convolution / Yosida regularization $(\tfrac1\beta |.|^2 \Box f)(x)=\inf_{y\in\R^d} \tfrac1\beta |x-y|^2 + f(y)$, the primal side is an infimal convolution with $\beta \W^2_2$.

We continue with discussing applications of \eqref{eq:RWMOTintro} to martingale transport:

\subsubsection{From Brenier--Strassen to martingale Benamou--Brenier}
Finding `natural' martingale transport plans between given marginals is an important problem in mathematical finance, e.g.\ to derive robust bounds or to obtain calibrated models. 
Recently, particular attention (see \cite{Lo18, BaBeHuKa20, GuLoWa19, CoHe21, AcMaPa23, BaLoOb24, BePaRi24} among others) has gone to \emph{Bass martingales}, i.e.\ martingales of the form 
\begin{align}\label{eq:Bass_definition}
    M_t=\EE[\nabla v (B_1)|B_t],
\end{align}
for some convex $v:\R^d\to \R$ and a Brownian motion $(B_t)_{t\in [0,1]}$ with possibly random $B_0$.
For sufficiently regular $\mu, \nu\in \P_2(\R^d), \mu\leq_c\nu$, there is a (unique) Bass martingale with $M_0\sim \mu, M_1\sim \nu$ which is further characterized as the solution to the martingale Benamou--Brenier problem, see \cite{BaBeHuKa20, BaBeScTs23}.\footnote{Alternatively, the Bass martingale is the martingale which is closest to Brownian motion in adapted Wasserstein distance subject to the given marginal constraints.}  
The key to this result is to recast the optimization problem as the WMOT problem 
\begin{equation*} \label{eq:Bass_primal} 
\sup_{\pi \in \cplm(\mu,\nu)} \int \MCov(\pi_{x},\gamma) \, \mu(dx),  \text{ where }\MCov(\rho,\gamma) := \sup_{q \in \cpl(\rho,\gamma)} \int  y\cdot z  \, q(dy,dz) 
\end{equation*}
and $\gamma$ is the standard Gaussian.

Given $\alpha, \beta >0$, we consider its `barycentric relaxation'  
\begin{equation}\label{eq:BS_SBM_WOT_intro}
     \inf_{\pi \in \cpl(\mu,\nu)}
    \int \beta |x - \mean{\pi_x}|^2 - \alpha \MCov(\pi_x,\gamma) \, \mu(dx),
\end{equation}
which has as extreme cases Brenier--Strassen (for $\alpha\to 0$) and the martingale Benamou--Brenier problem (for $\beta\to \infty$).

 One can rewrite \eqref{eq:BS_SBM_WOT_intro} as a relaxed WMOT problem so that dual attainment \eqref{eq:RWMOTintro} is applicable. In particular we obtain that there exist unique $T:\R^d\to \R^d$ and $\kappa\in \cplm(T_\#\mu, \nu)$ such that $\pi_x:= \kappa_{T(x)}$ is the unique optimizer of \eqref{eq:BS_SBM_WOT_intro}. Moreover, $T$ is $\tfrac1\beta$-Lipschitz, the gradient of a convex function and $\kappa=\law(M_0, M_1)$ for a Bass martingale $M$. 

We underline that this behavior is more regular than either of the extremes: the optimizer in Brenier--Strassen is not unique, while Bass martingales exist only under additional assumptions on $\mu\leq_c \nu$ in the classical MOT case. In Section \ref{sec:BS_SBM} we give a precise description of the optimizer to \eqref{eq:BS_SBM_WOT_intro} and further extensions.

\subsubsection{Relaxing martingale EOT}

We close this section with another example that becomes more regular by considering a relaxation of the martingale transport as in \eqref{eq:RWMOTintro}. Specifically, we consider the problem 
\begin{align}\label{eq:REMOT_intro}
   \min_{\eta\in \P_2(\R^d)}\Big( \beta \W_2^2(\mu, \eta)  +  \min_{\kappa\in \cplm(\eta, \nu)} \int c\, d\kappa + \eps H(\kappa| \eta\otimes \nu)   \Big)
\end{align}
where $c:\R^d\times \R^d\to [0,\infty)$ is lsc and $\beta, \eps>0$.

Here the usual martingale transport would correspond to the case $\beta\to\infty$, where \eqref{eq:REMOT_intro} is the natural entropic regularization of the martingale transport problem. However, due to the intricacies of MOT duality, this problem appears challenging and has so far been solved only for $d=1$, $c\equiv 0$ and regularity conditions for $\mu, \nu$, see \cite{NuWi24}. In contrast \eqref{eq:REMOT_intro} is more tractable due to strong duality in \eqref{eq:RWMOTintro}: If $c$ is Lipschitz and $\mu$ is absolutely continuous, we show there exist unique solutions $\eta, \kappa$; moreover 
 for suitable functions $\pphi, \ppsi, \Delta$, the density of $\kappa$ is of Gibbs type
\begin{align}
    \frac{d\kappa}{\eta\otimes \nu}  (\bar x,y) = \exp\left(- \frac{c(\bar x,y) + \pphi(\bar x) + \ppsi(y)+ \Delta(\bar x)(y-\bar x)}{\eps} \right).
\end{align}

\subsection{Duality and dual attainment without convexity assumptions}
\label{ssec:noconvexity}
Many known problems which fall in the weak transport framework exhibit a cost function such that $\rho\mapsto C(x, \rho)$ is convex and lsc. 
Moreover, in sufficiently regular settings, it is possible to investigate the weak transport problem for non-convex costs by considering the respective convex hulls. Specifically, assume that $\mu$ is continuous,  
$C$ is jointly continuous with bounded $p$-growth and define $\bar C$ such that $\bar C(x, \cdot)$ is the convex lsc hull of $C(x, \cdot)$ for every $x\in X$. Then \cite[Theorem 3.9]{AcBePa20} asserts that 
\begin{equation}\label{eq:reduceToConvex}
    \inf\limits_{\pi\in\cpl(\mu,\nu)}\int  C(x,\pi_x)\, \mu(dx) = \inf\limits_{\pi\in\cpl(\mu,\nu)}\int\bar C(x,\pi_x)\, \mu(dx).
\end{equation}
In particular, we recover the duality relation as well as dual attainment in this case. However while the right-hand side of \eqref{eq:reduceToConvex} is of course attained, the original primal problem on the left-hand side is not in general. 

In some settings it is important to have primal attainment, as well as dual attainment without imposing convexity or continuity of $\mu$ and $C$. For instance this is the case for the relaxed entropic MOT problem, in mathematical finance in connection to robust hedging of American options and for the model independent pricing problem in a fixed income market, see \cite{AcBePa20}. Here the solution is to consider a suitably enriched space of transport plans, we discuss this in Section \ref{sec:Relaxed_Primal_Problem}.
Using this relaxed formulation of the primal problem, we obtain a general fundamental theorem of WOT without imposing convexity of $\rho \mapsto C(x,\rho)$, see \Cref{thm:FTWOT_mainbody_L}.

\subsection{Related literature}

Weak transportation costs appeared in the works of Marton \cite{Ma96concentration, Ma96contracting} and Talagrand \cite{Ta95, Ta96}, where the authors studied the concentration of measure phenomenon via transport entropy inequalities. Motivated by questions of concentration and curvature properties of discrete measures, Gozlan, Roberto, Samson, and Tetali \cite{GoRoSaTe14} provided the general definition of weak transport costs and studied their basic properties. In particular the Kantorovich duality is obtained for costs which are convex in the second argument and satisfy some additional mild regularity conditions. Subsequently, Alibert, Bouchitte and Champion \cite{AlBoCh18} as well as Bowles and Ghoussoub \cite{BoGh18} introduced closely related setups, specifically \cite{AlBoCh18} also establishes primal existence and duality on compact state spaces. 
Chon\'e, Gozlan and Kramarz \cite{ChGoKr23} relax weak transport further to include also transports by unnormalized kernels, see also \cite{ScTo23}.

Based on using the adapted weak topology on the set of couplings, \cite{BaBePa19} extends existence and duality to the level of generality familiar from classical transport, that is lsc, $[0,\infty]$-valued costs on abstract Polish spaces. We remark that the formulation we give in Theorem \ref{thm:FTWOTintro} is slightly stronger in order to include also EOT in its usual generality. Existence and duality results beyond the convex costs are given in \cite{AcBePa20}, see also Sections \ref{ssec:noconvexity} and \ref{sec:Relaxed_Primal_Problem} below. In analogy to classical {$c$-cyclical monotonicity} \cite{BaBeHuKa20, GoJu18, BaBePa19} develop the notion of \emph{$C$-monotonicity} as a necessary criterion for optimality in weak transport. The article \cite{BaPa19} establishes sufficiency of this criterion as well as stability of WOT w.r.t.\ its marginals.

Besides the original motivation of geometric inequalities, WOT contains a variety of further problems that fall outside the class of classical optimal transport. Particular attention has gone to transport costs for barycentric cost functions
\cite{GoRoSaSh18, Sh16, Sa17, Sh18, FaSh18, GoJu18,  BaPa19, BaBePa19, AlCoJo17}.
As noted in \cite{Co19}, WOT covers entropic optimal transport EOT (see e.g.\ \cite{Le13, Nu22}). It includes martingale optimal transport MOT (see e.g.\ \cite{HeTo13, BeJu16}), semi-martingale optimal transport SMOT (see e.g.\ \cite{TaTo13, BeChHoLoVi24}) and the optimal Skorokhod embedding problem with non-trivial starting law (\cite{BeCoHu14}). WOT has been used to define so-called shadow (sub-) martingales (\cite{BeJu21, BaNo24}). It appears as important tool in the investigation of Bass martingales (e.g.\ \cite{BaBeHuKa20, BePaRi24}) and the recent probabilistic proof to the Caffarelli contraction theorem \cite{FaGoPr20}. Beyond its applications to martingale transport problems, WOT includes a number of further problems in economics and finance, we mention optimal mechanism design \cite{DaDeTz17}, optimal transport planning \cite{AlBoCh18}, stability of pricing and hedging \cite{BaBaBeEd19a}, model-independence in fixed income markets \cite{AcBePa20}, risk measures \cite{KuNeSg24}, and robust pricing of VIX futures \cite{GuMeNu17, CoVi19, BeJoMaPa21b}.
Furthermore, the WOT framework has been applied to machine learning tasks such as unpaired domain translation problems, generative modeling, and the learning of Wasserstein barycenters, see \cite{KoSeBu23,AsKoEgMoBu24,BrChHo22, KoMoUdGaPaBuKo24}.

\section{Fundamental theorem of weak optimal transport}\label{sec:dual}
The aim of this section is to prove the fundamental theorem of weak optimal transport. As already in the case of classical optimal transport (on non-compact spaces), we cannot expect dual attainment when maximizing only over continuous functions. For this reason, we need to introduce the dual problem in the correct generality. 
We denote by $\cL^1(\rho)$ the space of Borel functions that are $\rho-$integrable and possibly take the value $-\infty$.
We call $(\pphi,\ppsi)$ admissible if $\pphi \in \cL^1(\mu)$ and $\ppsi \in \cL^1(\nu)$ are such that
\begin{equation}
    \label{eq:adm}
    \pphi(x) + \rho(\ppsi) \le C(x,\rho)
\end{equation} 
for all $(x,\rho) \in X \times \P_p(Y)$ with $\ppsi \in \cL^1(\rho)$. The set of all admissible pairs is denoted by $\Dadm{C}{\mu}{\nu}$.
Note that since \eqref{eq:adm} has to hold pointwise, we do not identify $\pphi$ and $\ppsi$ with their respective $L^1$-equivalence classes.
The dual problem is then given by
\begin{equation}\tag{DP}
    \DW{C}{\mu}{\nu}:=\sup_{(\pphi,\ppsi) \in \Dadm{C}{\mu}{\nu} } \mu(\pphi) + \nu(\ppsi).
\end{equation}
We recall that the $C$-transform for weak transport costs is given by
\[
    \ppsi^C (x) := \inf \{ C(x,\rho) - \rho(\ppsi) : \rho \in \cP_p(Y) \text{ s.t.\ } \ppsi \in \cL^1(\rho) \},
\]
with the convention that the infimum over the empty set is $+\infty$. Note that $\ppsi^C$ is universally measurable, provided that $C$ and $\ppsi$ are measurable, see e.g.\ \cite[Proposition~7.47]{BeSh78}. 
 
To set the stage for the main result of this section, we introduce the following assumption under which we will work in most parts of this section. 
\begin{assumption}\label{asn}
Let $p \in [1,\infty)$, $\mu \in \cP(X)$ and $\nu \in \cP_p(Y)$. Assume that the cost function $C:X\times \P_p(Y)\to \R \cup \{+\infty\}$ is measurable, $\rho\mapsto C(x,\rho)$ is convex and lsc for the $p$-weak convergence on $\P_p(Y)$, and lower bounded in the sense that $C(x,\rho) \ge a_\ell(x) + \rho(b_\ell)$ for some $a_\ell \in \mathcal L^1(\mu), b_\ell \in \mathcal L^1(\nu)$. Moreover, we assume that $\WOT{C}{\mu}{\nu} < \infty$.

\end{assumption}

\begin{theorem}\label{thm:FTWOT_mainbody}
    Under \Cref{asn}, we have the following:
    
    \begin{enumerate}[\upshape(i)]
    \item\label{it:FTWOT_main1} (primal attainment) $\WOT{C}{\mu}{\nu}$ is attained.
    \item\label{it:FTWOT_main2} (duality) It holds $\WOT{C}{\mu}{\nu} = \DW{C}{\mu}{\nu}$ and
        \begin{align}
            \DW{C}{\mu}{\nu}= \label{eq:WOTdual2}
            \sup\{\mu(\ppsi^C)+\nu(\ppsi): \ppsi \in \mathcal L^1(\nu)  \}.
        \end{align}
    \item\label{it:FTWOT_main3} (dual attainment) If $C$ satisfies conditions \eqref{eq:WOTbound} and \eqref{eq:WOTcont},
    then ${\rm D}_C(\mu,\nu)$ is attained.
    \item\label{it:FTWOT_main4} (complementary slackness) A pair of candidates $(\pi, (\pphi, \ppsi))$ is optimal if and only if
    \begin{align}\label{eq:WOTCS2}
        C(x,\pi_x)= \pphi(x)+\pi_x(\ppsi), \quad \mu\text{-a.s.}
    \end{align}
    In this case, $\pphi(x) = \ppsi^C(x)$ $\mu$-a.s.
\end{enumerate}
\end{theorem}

Note that this is a slightly more general version of \Cref{thm:FTWOTintro} stated in the introduction. 

\begin{proof}
The primal attainment result and duality is proved in \Cref{thm:noGap} below. Dual attainment is shown in \Cref{thm:Attainment} below. Finally, the complementary slackness result follows from \Cref{prop:slack} below.
\end{proof}

\subsection{Primal attainment and duality}

The $C$-transform is a key tool for deriving structural properties of (dual) optimizers to the weak transport problem. Before proving duality, we derive some basic properties of the $C$-transform. 

\begin{lemma}\label{lem:C-trafo}
Let $C: X \times \cP_p(Y) \to [-\infty,+\infty]$ be Borel and suppose that $\WOT{C}{\mu}{\nu} < \infty$.
The $C$-transform has the following properties:
\begin{enumerate}[(i)]
    \item\label{lem:C-trafo_item1} For every $(\pphi,\ppsi)\in  \Dadm{C}{\mu}{\nu}$, we have $(\ppsi^C,\ppsi) \in \Dadm{C}{\mu}{\nu}$ and $\ppsi^C \ge \pphi$.
    \item\label{lem:C-trafo_item2} We can restrict the dual problem to admissible pairs of the form $(\ppsi^C,\ppsi)$, i.e.,
    \[
    \DW{C}{\mu}{\nu} = \sup \{ \mu(\ppsi^C) + \nu(\ppsi) : \ppsi \in \mathcal L^1(\nu)  \}
    \]
    \item\label{lem:C-trafo_item3} If $(\pphi,\ppsi)$ is a dual optimizer, so is $(\ppsi^C,\ppsi)$. In this case, $\pphi = \ppsi^C$ $\mu$-a.s.
\end{enumerate}
\end{lemma}

\begin{proof}
    To prove \eqref{lem:C-trafo_item1}, let $(\pphi,\ppsi) \in \Dadm{C}{\mu}{\nu}$.
    For $(x,\rho) \in X \times \P_p(Y)$ with $\ppsi \in \cL^1(\rho)$, we have
    \begin{equation}
        \label{eq:lem.C-trafo.1}
        \pphi(x) \le \ppsi^C(x) \le C(x,\rho) - \rho(\ppsi). 
    \end{equation}
    Since $\WOT{C}{\mu}{\nu} < \infty$, there exists a coupling $\pi \in \cpl(\mu,\nu)$ with $\int C(x,\pi_x) \, \mu(dx) < \infty$.
    Using that $\ppsi \in \cL^1(\nu)$ and the second inequality in \eqref{eq:lem.C-trafo.1}, we find that $g^C$ is $\mu$-a.s.\ dominated from above by the $\mu$-integrable function
    \begin{equation}
        \label{eq:lem.C-trafo.2}
        x \mapsto C(x,\pi_x) - \pi_x(\ppsi).
    \end{equation} 
    Together with $\pphi \in \cL^1(\mu)$ and the first inequality in \eqref{eq:lem.C-trafo.1}, this yields that $\ppsi^C \in  \cL^1(\mu)$.
    Hence, $(\ppsi^C,\ppsi) \in \Dadm{C}{\mu}{\nu}$.

    The assertion \eqref{lem:C-trafo_item2} follows directly from \eqref{lem:C-trafo_item1}.  
    
    To show \eqref{lem:C-trafo_item3}, let $(f,g)$ be a dual optimizer.
    Note that by $\eqref{lem:C-trafo_item1}$ the pair $(\ppsi^C,\ppsi)$ is admissible and $\mu(\pphi) + \nu(\ppsi) \le \mu(\ppsi^C) + \nu(\ppsi)$.
    Since $(\pphi,\ppsi)$ is optimal, so is $(\ppsi^C,\ppsi)$ and we have $\mu(\pphi) = \mu(\ppsi^C)$.
    Hence, as $\ppsi^C\geq\pphi$ and $\mu(\pphi) = \mu(\ppsi^C)$ we conclude that $\pphi = \ppsi^C$ $\mu$-a.s.
\end{proof}

\begin{lemma} \label{lem:weakDuality}
Let $C: X \times \cP_p(Y) \to [-\infty,+\infty]$ be Borel. We have $\DW{C}{\mu}{\nu} \le \WOT{C}{\mu}{\nu}$.
\end{lemma}
\begin{proof}
    Wlog assume that $\WOT{C}{\mu}{\nu}<+\infty$.
    Let $(\pphi,\ppsi)\in\Dadm{C}{\mu}{\nu}$ and $\pi\in\cpl(\mu,\nu)$ with finite cost.
    As in the proof of Lemma \ref{lem:C-trafo} \eqref{lem:C-trafo_item1} we find that $\mu$-a.s.
    \[
        \pphi(x) \le  C(x,\pi_x) - \pi_x(\ppsi).
    \]
    Integrating both sides w.r.t.\ $\mu$ and rearranging terms yield
    \[
        \mu(\pphi) + \nu(\ppsi) \le \int C(x,\pi_x) \, \mu(dx).
    \]
    Taking the infimum over $\pi \in \cpl(\mu,\nu)$ with finite cost and the supremum over $(\pphi,\ppsi) \in \Dadm{C}{\mu}{\nu}$ yields the claim.
\end{proof}

The first step of the proof of \Cref{thm:FTWOT_mainbody} is to derive the no duality gap result in \Cref{thm:noGap}. Note that this extends \cite[Theorem 3.1]{BaBePa19} by requiring only that $C$ is convex lsc in the second component but not jointly lsc. This generalization is relevant e.g.\ as it allows us to derive the structure theorem of entropic optimal transport in its usual generality in Section \ref{sec:reg_OT}.

\begin{theorem}\label{thm:noGap}
Under \Cref{asn}, we have primal attainment of $\WOT{C}{\mu}{\nu}$ and 
    \begin{align*}
        \WOT{C}{\mu}{\nu} = \DW{C}{\mu}{\nu}
        =
        \sup\{\mu(\ppsi^C)+\nu(\ppsi): \ppsi \in \mathcal L^1(\nu) \}.
    \end{align*}
\end{theorem}

\begin{proof}
    We start with some preparations before proving the assertion of the theorem.
    First, note that we may wlog assume $b_\ell \le 0$.
    Next, write $d_Y$ for the (complete, separable) metric on $Y$ which is compatible with the Polish topology on $Y$, denoted by $\tau_Y$.
    By \cite[Theorem 13.11]{Ke95} there exists a Polish topology $\tilde \tau_Y\supseteq \tau_Y$ such that $b_\ell$ is $\tilde \tau_Y$-continuous and $\tau_Y,\tilde\tau_Y$ induce the same Borel $\sigma$-algebra.
    Let $\tilde d_Y$ be a bounded, complete, separable metric on $Y$ that is compatible with $\tilde \tau_Y$.
    Define the metric $\hat d_Y(y_1,y_2) := \max( d_Y(y_1,y_2), \tilde d_Y(y_1,y_2))$, which is complete, separable and compatible with $\tilde \tau_Y$.
    Denote the associated $p$-Wasserstein distance by $\hat \W_p$.
    Since $\W_p \le \hat \W_p$ and as $b_\ell$ is continuous and non-positive, we find that
    \[
        \tilde C(x,\rho) := C(x,\rho) - a_\ell(x) - \rho(b_\ell),
    \]
    is measurable, non-negative and $\rho \mapsto \tilde C(x,\rho)$ is convex and $\hat \W_p$-lsc.
    We proceed to show the assertion for the cost $\tilde C$ and remark that then the claim also follows for the cost $C$.
    
    We endow $X \times Y$ with the metric $d((x,y),(x',y')) = ((d_X(x,x') \wedge 1)^p + \hat d_Y(y,y')^p)^\frac1p$.
    Consider
    \[
        \Pi := \{ \pi \in \P_p(X \times Y) : \pi(dx \times Y) = \mu(dx) \},
    \]
    and observe that weak convergence and stable convergence coincide on $\Pi$, by \cite[Lemma A.10]{BeJoMaPa21b}.
    Hence, \cite[Proposition A.12 (d)]{BeJoMaPa21b} yields that
    \[
        \Pi \ni \pi \mapsto I_{\tilde C}[\pi] := \int \tilde C(x,\pi_x) \, \mu(dx),
    \]
    is lsc for the $p$-weak convergence on $\P_p(X \times Y)$.
    In particular, if $(\nu_n)_{n \in \N}$ converges to $\nu$ in $\hat \W_p$ and $(\pi_n)_{n \in \N}$ is a sequence of couplings with $\pi_n \in \cpl(\mu,\nu_n)$, the latter sequence is even relatively compact in $\P_p(X \times Y)$.
    Consequently, by potentially passing to a subsequence, we may assume that $\pi_n \to \pi$ in $p$-convergence on $\P_p(X \times Y)$ where $\pi \in \cpl(\mu,\nu)$.
    Assuming that $\pi_n$ was $\frac1n$-optimal for $\WOT{{\tilde C}}{\mu}{\nu_n}$ we obtain that
    \[
        \liminf_{n \to \infty} \WOT{{\tilde C}}{\mu}{\nu_n} =
        \liminf_{n \to \infty} I_{\tilde C}[\pi_n] \ge
        I_{\tilde C}[\pi] \ge \WOT{{\tilde C}}{\mu}{\nu}.
    \] 
    
    Hence, $\nu \mapsto \WOT{{\tilde C}}{\mu}{\nu}$ is $\hat \W_p$-lsc and it also follows (by letting $(\pi_n)_{n \in \N}$ be a minimizing sequence in $\cpl(\mu,\nu$) that $\WOT{{\tilde C}}{\mu}{\nu}$ is attained.
    Using lower semicontinuity of $\nu \mapsto \WOT{{\tilde C}}{\mu}{\nu}$, we can follow line by line \cite[Proof of Theorem 3.1]{BaBePa18} and obtain
    \[
        \WOT{{\tilde C}}{\mu}{\nu} = 
        \sup \mu(\ppsi^{\tilde C}) + \nu(\ppsi),
    \]
    where the supremum is taken over all $\ppsi \in C((Y,\hat d_Y))$ with at most $p$-growth that are bounded from above.
    In particular, $\ppsi \in \cL^1(\nu)$ and $\mu(\ppsi^{\tilde C})$ is well-defined with value in $[-\infty,+\infty)$.
    Now, the assertion follows by the first part of the proof.
\end{proof}

\subsection{Complementary slackness} \label{ssec:CS}

\begin{proposition}\label{prop:slack}
Suppose that \Cref{asn} is satisfied. Then, $\pi\in\cpl(\mu,\nu)$ and $(\pphi,\ppsi) \in \Dadm{C}{\mu}{\nu}$ are both optimal if and only if they satisfy the complementary slackness condition 
\begin{align}\label{eq:WOTCS3}
    C(x,\pi_x)= \pphi(x)+\pi_x(\ppsi), \quad \mu\text{-a.s.}
\end{align}      
\end{proposition}
\begin{proof}
    Let $\pi\in\cpl(\mu,\nu)$ and $(\pphi, \ppsi)\in\Dadm{C}{\mu}{\nu}$ both be optimal. As $g\in\cL^1(\nu)$, we have
    \begin{equation*}
        \infty>\int |\ppsi(y)|\, \nu(dy) =\int_X\int_Y|\ppsi(y)|\,\pi_x(dy)\,\mu(dx),
    \end{equation*}
    which implies $\ppsi \in \cL^1(\pi_x)$ $\mu$-a.s. 
    By admissibility of $(\pphi,\ppsi)$ we have $\pphi(x)+\pi_x(\ppsi)\leq C(x,\pi_x)$ $\mu$-a.s. Since $\pi$ and $(\pphi,\ppsi)$ are both optimal, it follows
    \begin{equation*}
        \int C(x,\pi_x)-\pphi(x)-\pi_x(\ppsi) \,\mu(dx) = 0,
    \end{equation*}
    implying $\pphi(x)+\pi_x(\ppsi)=C(x,\pi_x)$ $\mu$-a.s.

    Suppose now that $\pi\in\cpl(\mu,\nu)$ and $(\pphi,\ppsi)\in\Dadm{C}{\mu}{\nu}$ satisfy \eqref{eq:WOTCS3}.
    By \Cref{lem:weakDuality} we have
    \[
        \mu(\pphi) + \nu(\ppsi) \le {\rm WT}_C(\mu,\nu) \text{ and } {\rm D}_C(\mu,\nu) \le \int C(x,\pi_x) \, \mu(dx).
    \]
    Hence, the assertion follows as $\int C(x,\pi_x) \, \mu(dx) = \mu(\pphi) + \nu(\ppsi)$.
    
\end{proof}

\subsection{Dual attainment}

The last major step in the proof of the fundamental theorem of weak optimal transport is the existence of dual optimizers.

\begin{theorem}\label{thm:Attainment}
Suppose that $C: X \times \cP_p(Y) \to \R$ satisfies conditions \eqref{eq:WOTbound} and \eqref{eq:WOTcont}. Then there is a dual optimizer $(\pphi,\ppsi) \in\Dadm{C}{\mu}{\nu}$. 
\end{theorem}

In order to prove \Cref{thm:Attainment}, we need the following growth estimate on admissible potentials, which will yield uniform integrability of the maximizing sequence.
    
\begin{lemma}\label{lemma:uniform_integrability_positive_part_max_seq}
Let $h: [0,\infty) \to [0,\infty)$ be convex, increasing and super-coercive, $b \in \mathcal{L}^1(\nu)$ and $\alpha \in \R$. Then, there exist constants $K_1,K_2$ (depending on $\alpha$, $b$ and $h$) such that the following holds: If $\ppsi \in \mathcal{L}^1(\nu)$ satisfies $\int \ppsi \,d\nu = 0$ and 
\begin{align}\label{eq:lemUIasn}
    \int (\ppsi-b) \xi \,d\nu \le \alpha + \int h(\xi) \,d\nu
\end{align}
for every $\xi \in \mathcal{L}^\infty(\nu)$ with $\xi \ge 0$ and $\int \xi \,d\nu =1$, then we have 
\begin{align}\label{eq:lemUIconclusion}
    \int h^\ast ((\ppsi - b)^+-K_1) \,d\nu\le K_2.
\end{align}
\end{lemma}
\begin{proof}
We assume that $\nu$ has no atoms. This is without loss of generality and will simplify the notation in the proof. Further, we can assume wlog that $\alpha=0$ and $\int b \, d\nu =0$. Otherwise replace $h$ with $h + \alpha_+ + (\int b \,d\nu)_+$. As $(h + \alpha_+ + (\int b \,d\nu)_+)^\ast = h^\ast - \alpha_+ - (\int b \,d\nu)_+$, this replacement just causes a change of the constants $K_1$ and $K_2$ in \eqref{eq:lemUIconclusion}. Moreover, we introduce the shorthand notation $\tilde \ppsi := \ppsi -b$. 

The first step is to derive a bound on $\int  \tilde \ppsi^+ \,d\nu$. 
To that end, we distinguish two cases depending on the value of $\beta := \nu( \tilde \ppsi \ge 0 )$: In the case that $\beta \ge \frac{1}{3}$, we choose the test function $\xi = \beta^{-1} 1_{ \{ \tilde \ppsi \ge 0\} }$ in \eqref{eq:lemUIasn}. Using that $h$ is increasing we find
\[
\int \tilde{\ppsi}^+\, d\nu =  \beta \int \tilde{\ppsi} \xi \,d\nu  \le \beta \int h(\xi) \,d\nu = \beta(1-\beta) h(0) + \beta^2 h(1/\beta) \le  h(3). 
\]
To tackle the case that $\beta <\frac{1}{3}$, fix $t <0$ and $A \subseteq Y$ such that $\nu(A)=\frac{1}{3}$ and $\{ \tilde{\ppsi} \in (t,0) \} = A$.
Further, write $B := \{\tilde \ppsi < 0\} \setminus A$ and note that $\nu(B) \ge 1- \beta - \frac{1}{3} \ge \frac{1}{3} = \nu(A)$. As $\tilde \ppsi \le t$ on $B$, $\tilde \ppsi \ge t$ on $A$ and $\int \tilde\ppsi \, d\nu =0$, we find
\[
\int \tilde \ppsi^+ \, d\nu = - \int_{A \cup B} \tilde \ppsi \, d\nu \ge - 2 \int_A \tilde \ppsi \,d\nu.
\]
By using the test function $\xi = (\beta +\frac{1}{3})^{-1} 1_{B^c}$ in \eqref{eq:lemUIasn} we derive
\begin{align*}
    \int \tilde{\ppsi}^+\, d\nu  &\le 2 \left( \int \tilde{\ppsi}^+\, d\nu + \int_A \tilde{\ppsi}\, d\nu  \right)  = 2\left(\beta + \frac{1}{3}\right) \int \tilde \ppsi \xi \, d\nu \\
    &\le2\left(\beta + \frac{1}{3}\right) \int h(\xi) \, d\nu \le 2\left(\beta + \frac{1}{3}\right) h(3) \le 2h(3).
\end{align*}
Combining both cases and using again that $\int \tilde \ppsi \,d\nu =0$, we find
\begin{align}\label{eq:bound_psiminus}
    \| \ppsi^- \|_{L^1(\nu)} = \| \ppsi^+ \|_{L^1(\nu)} \le 2h(3).
\end{align}
In the next step, we show that for every non-negative, bounded Borel function $\xi$ 
\begin{align}\label{eq:lemUIbound2}
\int \tilde \ppsi^+ \xi \,d\nu \le \int h(\xi) + h(0)\xi \, d\nu + 3h(3).   
\end{align}
To this end, we distinguish two cases:

Case 1: If $\| \xi 1_{ \{ \tilde \ppsi \ge 0 \} } \|_{L^1(\nu)} \ge 1$, we set $\tilde \xi = \| \xi 1_{ \{ \tilde \ppsi \ge 0 \} } \|_{L^1(\nu)}^{-1} \xi 1_{ \{ \tilde \ppsi \ge 0 \} }$. By first applying \eqref{eq:lemUIasn} with the test function $\tilde \xi$ and then using that $h$ is increasing, convex, and non-negative, we find
\begin{align*}
    \int \tilde \ppsi^+ \xi \, d\nu  &=  \| \xi 1_{ \{ \tilde \ppsi \ge 0 \} } \|_{L^1(\nu)} \int \tilde \ppsi \tilde \xi \, d\nu \le \| \xi 1_{ \{ \tilde \ppsi \ge 0 \} } \|_{L^1(\nu)} \int h(\tilde \xi)  \, d\nu\\
    &\le \| \xi 1_{ \{ \tilde \ppsi \ge 0 \} } \|_{L^1(\nu)} h(0) + \int h(\xi) \, d\nu
    \le \int  h(\xi) +h(0)\xi \, d\nu.
\end{align*}

Case 2: If $\| \xi 1_{ \{ \tilde \ppsi \ge 0 \} } \|_{L^1(\nu)} < 1$, there is $r \in (0,1]$ such that $\tilde \xi := (\xi 1_{\{\tilde \ppsi \ge 0\}}) \vee r $ satisfies $\| \tilde \xi \|_{L^1(\nu)} = 1$. Applying \eqref{eq:lemUIasn} with the test function $\tilde \xi$, equation \eqref{eq:bound_psiminus}, and the fact that $h$ is increasing convex, yields the estimate
\begin{align*}
\int \tilde \ppsi^+ \xi \,d\nu \le  \int \tilde \ppsi \tilde \xi \,d\nu + r \| \tilde \ppsi^- \|_{L^1(\nu)} \le \int h(\tilde \xi) \,d\nu + 2r h(3) \le \int h(\xi) \, d\nu + 3h(3).
\end{align*}
This completes the proof of \eqref{eq:lemUIbound2}.

The next step is to use \eqref{eq:lemUIbound2} to derive a bound on $\int h^\ast(\tilde\ppsi^+ -h(0)) \, d\nu$.
Formally by exchanging integration with supremum, we have
\begin{align*}
    \int h^\ast(\tilde \ppsi^+-h(0)) \, d\nu &= \int \sup_{z \in [0,\infty)} \tilde\ppsi^+ (y) z - h(0)z -h(z) \, \nu(dy) \\
    &= \sup_{\xi \ge 0 \text{ Borel}} \int \tilde\ppsi^+(y) \xi(y) -  h(0)\xi(y) - h(\xi(y)) \, \nu(dy) \le 3h(3).
\end{align*}
In order to rigorously justify this calculation, let $g$ be a selector of $\partial h^\ast$, i.e., $g(t) \in \partial h^\ast (t)$ for every $t \in [0,\infty)$. This selector is well-defined (as $h$ is super-coercive, $\dom(h^\ast)=[0,\infty)$ and thus $\partial h^\ast(t) \neq \emptyset$ for every $t$), monotone and in particular Borel. Moreover, we have $tg(t)= h(g(t)) + h^\ast(t)$. For every $n \in \N$, let $\xi_n = g((\tilde \ppsi^+ -h(0)) \wedge n) $ and note that $0 \le \xi_n \le g(n)$, so we can apply \eqref{eq:lemUIbound2} and get 
\begin{align*}
    \int h^\ast ((\tilde \ppsi^+ -h(0)) \wedge n) \,d\nu = \int ((\tilde \ppsi^+ -h(0)) \wedge n) \xi_n - h(\xi_n) \, d\nu \le 3h(3).
\end{align*}
The integral $\int [h^\ast (\tilde \ppsi^+-h(0))]^- \,d\nu$ is well-defined, because $\tilde\ppsi^+ \in \cL^1(\nu)$ and $h^\ast$ is convex and thus bounded from below by an affine function. 
As $h^\ast$ is increasing (cf. \Cref{sec:app1}), we can therefore use monotone convergence to obtain that $\int h^\ast (\tilde \ppsi^+ -h(0)) \,d\nu  \le 3h(3)$. 
\end{proof}

\begin{proof}[Proof of \Cref{thm:Attainment}]
    Let $(\pphi_n,\ppsi_n)_n$ be a maximizing sequence for the dual problem, that is, $(\pphi_n,\ppsi_n) \in \Dadm{C}{\mu}{\nu}$ and $\mu(\pphi_n) + \nu(\ppsi_n) \nearrow \DW{C}{\mu}{\nu}.$ 
    Replacing $(\pphi_n,\ppsi_n)$ by $(\pphi_n+c,\ppsi_n-c)$ where $c \in \R$, leaves value and admissibility unchanged.
    Therefore, we can assume that $\nu(\ppsi_n)= 0$. 
    By applying the admissibility condition \eqref{eq:adm} with $\rho = \nu$, we find the upper bound
    \begin{equation*}
        \pphi_n(x)\leq C(x,\nu) \leq a(x)+\int b \, d\nu + h(1).
    \end{equation*}
    Hence, integrating w.r.t.\ $\mu$ leads to
    \begin{equation*}
        \int \pphi_n \, d\mu \leq \int a \, d\mu + \int b \, d\nu + h(1),
    \end{equation*}
    which shows that $(\pphi_n)_n$ is bounded in $L^1(\mu)$. 
    By Koml\'os' lemma \cite{Ko67}, there is a sequence of convex combinations $(\hat \pphi_n)_n$ of $(\pphi_n)_n$ such that $\hat{\pphi}_n \to \pphi$ $\mu$-a.s.\ and $\pphi \in \cL^1(\mu)$.
    By potentially renaming the sequence, we assume wlog that $\pphi_n \to \pphi$ $\mu$-a.s.
    
    Next, fix some $x_0\in X$ for which $\pphi_n(x_0)\to \pphi(x_0)$. 
    As $(\pphi_n,\ppsi_n)$ is admissible, we get
    \begin{equation*}
        \int \ppsi_n \, d\rho \leq C(x_0,\rho) - \pphi_n(x_0)\leq a(x_0)-\pphi_n(x_0)+\int b \, d\rho + \int h\Big(\frac{d\rho}{d\nu}\Big)\, d\nu,
    \end{equation*}
    for every $\rho \in \P_p(Y)$ with $\|\frac{d\rho}{d\nu}\|_{L^\infty(\nu)} < \infty$.
    As $(\pphi_n(x_0))_n$ is convergent, there is $\alpha \in \R$ such that
    \begin{equation}\label{eq:inequality_psi_n-b}
        \int \ppsi_n - b\, d\rho \leq \alpha + \int h\Big(\frac{d\rho}{d\nu}\Big)\, d\nu.
    \end{equation}
    Hence, for every bounded $\xi\geq 0$ with $\int \xi \, d\nu=1$, we have
    \begin{equation}
        \int (\ppsi_n-b)\xi \, d\nu \leq \alpha + \int h(\xi)\, d\nu.
    \end{equation}
    
    In case that $h$ is not supercoercive, we may replace $h$ by $h + |\cdot|^2$.
    By Lemma \ref{lemma:uniform_integrability_positive_part_max_seq}, there exist constants $K_1, K_2$ such that $\int h^\ast((\ppsi_n -b)^+-K_1) \, d \nu \le K_2 $ for every $n \in \N$. 
    By \Cref{lem:superlinear}, $h^\ast$ is super-coercive, thus the de La Vallée Poussin criterion for uniform integrability yields that the sequences $((\ppsi_n -b)^+-K_1)_n$ and, in particular, $(\ppsi_n ^+)_n$ are uniformly integrable.

    As $\int \ppsi_n\, d\nu=0$, the sequence $(\ppsi_n)_n$ is bounded in $L^1(\nu)$. Therefore, Koml\'os' lemma guarantees the existence of $g \in L^1(\nu)$ and a sequence $(\hat g_n)_n$ consisting of forward convex combinations of $(g_n)_n$ such that $\hat g_n \to g$ $\nu$-a.s. To simplify notation, we rename the sequence $(\hat g_n)_n$ to $(g_n)_n$. Using Fatou's Lemma we obtain
    \begin{equation}
        \int \pphi \, d\mu + \int \ppsi \, d\nu \geq \limsup\limits_{n}\int \pphi_n \, d\mu + \limsup\limits_{n} \int \ppsi_n \, d\nu = \DW{C}{\mu}{\nu}.
    \end{equation}
    
    It remains to show that the pair $(\pphi,\ppsi)$ is admissible. By Egorov's theorem, there is an increasing sequence $(K_N)_N$ of compact subsets of $Y$ such that $\nu(\bigcup_N K_N)=1$ and $\ppsi_n\to\ppsi$ uniformly on $K_N$ for every $N \in \N$.
    We redefine $\pphi$ to be $-\infty$ on the $\mu$-null set where $(\pphi_n)_n$ does not converge to $\pphi$.
    Similarly, we change $\ppsi$ to $-\infty$ on the $\nu$-null set $(\bigcup_N K_N)^C$. 

    To conclude that $(f,g)$ is an optimal dual pair, it suffices to show that $(f,g)$ is admissible. Specifically, we need to show that for $(x,\rho) \in X \times \P_p(Y)$ satisfying $f(x) \in \R$ and $\ppsi \in L^1(\rho)$ we have
    \[
        \pphi(x)+ \rho(\ppsi) \leq C(x,\rho).
    \]
   
   First suppose that there is an $N \in \N$ such that $\rho(K_N)=1$. As $g \in \cL^1(\nu)$ and $g_n \to g$ uniformly on a set of full measure, we have $g_n \in \cL^1(\nu)$ and 
    \begin{equation*}
    \pphi(x)+ \rho(\ppsi) = \lim\limits_{n}\pphi_n(x)+\rho(\ppsi_n) \le C(x,\rho).
    \end{equation*}

    Suppose next that $\rho(\bigcup_N K_N)=1$. Define $\rho_N := \frac{1}{\rho(K_N)}\rho_{\mid K_N}$ and note that $\ppsi \in \cL^1(\rho_N)$. 
    By the dominated convergence theorem, we have $\rho_N(\ppsi)\to \rho(\ppsi)$. 
    Invoking the continuity property \eqref{eq:WOTcont} yields
    \[
        \pphi(x) + \rho(\ppsi) = \lim_N \pphi(x) + \rho_N(\ppsi)
        \le \limsup_N C(x,\rho_N) \le C(x,\rho).
    \]

 In the case that $\rho(\bigcup_N K_N)<1$, we have $g \notin \cL^1(\rho)$, so there is nothing to show. 
\end{proof}

\begin{remark}
If the cost function $C$ satisfies the stronger growth condition $
C(x,\rho) \le a(x) + \rho(b),
$ for some $a \in L^1(\mu)$, $b \in L^1(\nu)$, then the argument for the uniform integrability of the maximizing sequence $(\pphi_n,\ppsi_n)_n$ in the proof of Theorem~\ref{thm:Attainment} simplifies, namely there exists a constant $\alpha$ such that $\pphi_n \le \alpha + a$ and $\ppsi_n \le \alpha +b$. In particular, there are dual maximizers $(\pphi,\ppsi)$ which satisfy $\pphi\le \alpha + a$ and $\ppsi \le \alpha +b$. \hfill$\diamond$
\end{remark}

\begin{example}[Necessity of assumption \eqref{eq:WOTcont}]\label{exmpl:asmp_WOTcont}
    Even in the case of a one element space $X$ and compact $Y$ we find a bounded, convex and lsc cost function that does not satisfy assumption \eqref{eq:WOTcont} and 
    for which we have no dual attainment. Let $X=\{0\},Y=[0,1]$ and take $\sigma\in\cP([0,1])$ with $\sigma\ll\lambda$ and $\lVert\frac{d\sigma}{d\lambda}\rVert_\infty = \infty$, where $\lambda$ denotes the Lebesgue measure. Consider the cost function
    \begin{align}
        C(0,\rho):=C(\rho):=1-\frac{1}{\lVert\frac{d\sigma}{d\rho}\rVert_\infty},
    \end{align}
    where we convene that $\lVert\frac{d\sigma}{d\rho}\rVert_\infty = \infty$ if $\sigma$ is not absolutely continuous w.r.t.\ $\rho$.
    Next, consider Borel sets $Y_n\subset[0,1]$ with $\sigma(Y_n)<1$ and $\bigcup_n Y_n = [0,1]$. As $\sigma$ is not absolutely continuous w.r.t. $\sigma_n$, we have $C(\tfrac{1}{\sigma(Y_n)}\sigma_{|Y_n}) = 1 > 0 = C(\sigma),$
showing that $C$ does not satisfy Assumption \eqref{eq:WOTcont}.

    Assume for the sake of contradiction that we have dual attainment, i.e., there is $\ppsi:[0,1]\to[-\infty,\infty)$ satisfying $\int\ppsi\, d\lambda = 1=\WOT{C}{\delta_0}{\lambda}$ and $\int\ppsi\,d\rho\leq C(\rho)$ for all $\rho\in \cP([0,1])$ with $\ppsi \in L^1(\rho)$. 
    Choosing $\rho=\delta_y$ implies $\ppsi(y)\leq C(\delta_y)=1$ for every $y\in[0,1]$ with $\ppsi(y) > -\infty$ and, thus, $\ppsi \le 1$. 
    Next, choosing $\rho = \sigma$ shows that $\int\ppsi\, d\sigma\leq C(\sigma) = 0$. Consider the set $A:=\{-\infty < \ppsi<1\}$. 
    On the one hand, $\int\ppsi\, d\lambda = 1$ as well as $\ppsi\leq 1$ imply $\lambda(A)=0$. 
    On the other hand, $\int\ppsi\, d\sigma \leq 0$ implies $\sigma(A)>0$. This is a contradiction to $\sigma \ll \lambda$. Therefore, there are no dual optimizers which shows that assumption \eqref{eq:WOTcont} cannot be omitted in \Cref{thm:FTWOTintro}. \hfill$\diamond$
\end{example}

We conclude this part with a result that guarantees the existence of particularly regular dual maximizers, provided that $C$ is Lipschitz in its second argument.

\begin{corollary}
Suppose that \Cref{asn} is satisfied with $p=1$ and that $C(x,\cdot)$ is $L$-Lipschitz w.r.t.\ $\mathcal W_1$-distance. Then we have
    \[
        \WOT{C}{\mu}{\nu} =
        \max_{\substack{(\pphi,\ppsi) \in \Dadm{C}{\mu}{\nu} \\ \ppsi \textrm{ $L$-Lipschitz}}} \mu(f) + \nu(g). 
    \]
\end{corollary} 
\begin{proof}
    Let $(\pphi,\ppsi) \in \Dadm{C}{\mu}{\nu}$ and define the $L$-Lipschitz function $\tilde \ppsi$ by
    \[
        \tilde \ppsi(y) = \sup_{z \in  Y} \ppsi(z) - L d_ Y(y,z).
    \]
    Fix $\epsilon > 0$ and pick a measurable selector off $T$ with $\tilde \ppsi(y) \le \ppsi(T(y)) - L d_Y(y,T(y)) + \epsilon$. 
    Note that $\ppsi \le \tilde \ppsi$.
    Therefore, the claim follows immediately from \Cref{thm:Attainment} if we can show that $(\pphi,\tilde\ppsi)$ is also admissible.\footnote{Indeed in the present Lipschitz case we could avoid Theorem \ref{thm:Attainment} by invoking an Arzel\`a--Ascoli argument.}
    To this end, fix $(x,\rho) \in  X \times \mathcal P_1( Y)$ and write $\tilde \rho = T_\# \rho$.
    Using admissibility of $(\pphi,\ppsi)$ we get
    \begin{align*}
        \pphi(x) + \rho(\tilde \ppsi) \le
        \pphi(x) + \tilde \rho(\ppsi) - L\mathcal W_1(\rho,\tilde \rho) + \epsilon 
        \le C(x,\tilde \rho) - L\mathcal W_1(\rho,\tilde \rho) + \epsilon \le C(x,\rho) + \epsilon.
    \end{align*}
    As $\epsilon > 0$ was arbitrary, $(f,\tilde g)$ is admissible.
\end{proof}

\subsection{$C$-monotonicity of optimal couplings}
The dual attainment result allows us to easily recover, under the current set of assumptions, the result from \cite[Theorem~5.3]{BaBePa18} that $C$-monotonicity is a necessary optimality criterion.

\begin{definition}
A set $\Gamma \subset X \times \cP(Y)$ is called $C$-monotone if for every $(x_1,\rho_1),\dots, (x_n,\rho_n) \in \Gamma$ and all $\tilde{\rho}_1, \dots, \tilde{\rho}_n\in \cP(Y)$ with $\sum_{i=1}^n \rho_i = \sum_{i=1}^n \tilde{\rho}_i$ it holds 
\[
\sum_{i=1}^n C(x_i,\rho_i) \le \sum_{i=1}^n C(x_i, \tilde{\rho}_i).
\]
A coupling $\pi \in \cpl(\mu,\nu)$ is called $C$-monotone if $\mu(\{x \in X : (x, \pi_x) \in \Gamma \})=1$.
\end{definition}

\begin{corollary} \label{cor:Cmonotonicity}
Suppose that \Cref{asn}, \eqref{eq:WOTbound} and \eqref{eq:WOTcont} are satisfied. Then, every optimal $\pi \in \cpl(\mu,\nu)$ is $C$-monotone. 
\end{corollary}
\begin{proof}
Let $\pi$ be a primal and $(\pphi, \ppsi)$ a dual optimizer. Set 
$$
\Gamma := \{ (x,\rho) \in X \times \cP_p(Y) : C(x,\rho) = \pphi(x) + \rho(\ppsi)\}.
$$ 
The complementary slackness criterion (see \Cref{prop:slack}) yields $\mu(\{x \in X : (x, \pi_x) \in \Gamma \})=1$. To see that $\Gamma$ is $C$-monotone, let $(x_1,\rho_1),\dots, (x_n,\rho_n) \in \Gamma$ and let $\tilde{\rho}_1, \dots, \tilde{\rho}_n\in \cP_p(Y)$ be competitors satisfying $\sum_{i=1}^n \rho_i = \sum_{i=1}^n \tilde{\rho}_i$. 
Then,
\[
    \sum_{i=1}^n C(x_i,\rho_i) = \sum_{i=1}^n \pphi(x_i) + \rho_i(\ppsi) = \sum_{i=1}^n \pphi(x_i) + \tilde\rho_i(\ppsi)\le \sum_{i=1}^n C(x_i, \tilde{\rho}_i). \qedhere
\]
\end{proof}

We notice that $C$-monotonicity requires relatively strong assumptions in order to be a sufficient condition for optimality. Indeed, even if we are in the classical ``linear'' OT case $C(x,\rho) = \int c(x,y) \, \rho(dy)$, $C$-monotonicity does not imply $c$-cyclical monotonicity see \cite[Example 5.2]{BaPa20}. Moreover, while $C$-monotonicity as a necessary condition for optimality is used for instance in \cite{BaBeHuKa20, GoJu18}, we are not aware of applications of $C$-monotonicity as a sufficient condition in WOT and do not pursue it further in this paper.

 \subsection{Non convex costs and duality for relaxed WOT} \label{sec:Relaxed_Primal_Problem}

 As noted in the introduction, duality can fail if $C$ is not convex in the second argument, even in otherwise very regular situations:
 \begin{example}\label{ex:NoDuality}
 Let $(X, \mu)=(\{0\}, \delta_0)$ and $ (Y, \nu)=(\{-1,+1\}, (\delta_{-1}+\delta_{+1})/2)$. Consider the cost $C(0, m \delta_{-1} + (1-m) \delta_{+1})= 1-|2m -1|$. Then the primal value is $1$ while the dual value is $0$.\hfill$\diamond$
\end{example} 
 A possible remedy is to suitably enrich the space of admissible transport plans, by allowing for initial randomization. This can be done either in measure theoretic or in probabilistic language:

 \begin{enumerate}
\item  \emph{Probabilistically}, the primal weak transport problem can be formulated as \begin{align}\label{eq:probFormulation}
{\rm WT}_C(\mu,\nu)  = \inf_{X_1\sim\mu, X_2\sim \nu}\EE [C(X_1, \law(X_2| X_1))] .\end{align}
The right-hand side of \eqref{eq:probFormulation} can be relaxed by allowing to take the minimum over all adapted processes $(X_t)_{t=1,2}$ defined on some stochastic basis $(\Omega, \F, \mathbb{P}, (\F_t)_{t=1,2})$, where $\F_1$ could be larger than $\sigma(X_1)$, i.e.
\[ {\overline{{\rm WT}}}_C(\mu, \nu):= \inf_{\substack{(\Omega, \F, \mathbb{P}, (\F_t)_{t=1}^2,(X_t)_{t=1}^2), \\ X_1\sim\mu, \, X_2\sim \nu}}\EE [C(X_1, \law(X_2| \F_1))]. \]

\item For the \emph{measure theoretic} formulation, we introduce the set of \emph{lifted transport plans}. Specifically, for $\mu \in \cP_p(X), \nu \in \cP_p(Y)$, we write $\Lambda(\mu, \nu)$ for the collection of $P \in \P_p(X\times \P_p(Y))$ for which the first marginal is $\mu$ (i.e.\ ${\textrm{pr}_{X}}_\# P= \mu$) and the intensity of the second marginal is $\nu$ (i.e.\ $\int \!\rho(g)\, ({\textrm{pr}_{\cP(Y)}}_\#P)(d\rho) = \nu(g)$ for every $g\in C_b(Y)$). 
Using this notion, the relaxed weak optimal transport problem can be written as 
\begin{align}\label{eq:def:WTR}
    {\overline{{\rm WT}}}_C(\mu, \nu)= \inf_{Q \in \Lambda(\mu, \nu)}\int C(x,\rho)\, Q(dx, d \rho).
\end{align}
Transport plans can be identified with lifted transport plans using the map
\begin{align}\label{eq:def:J}
    J: \cpl(\mu, \nu) \to \Lambda(\mu, \nu) ,\quad J(\pi) :=( x\mapsto (x,\pi_x))_{\#}\mu.
\end{align}
\end{enumerate}
As every transport plan $\pi$ induces a lifted plan $J(\pi)$, we always have 
\begin{align}\label{eq:RelaxIneq}
    \overline{\rm WT}_C(\mu,\nu) \le {\rm WT}_C(\mu,\nu). 
\end{align}
Note that lifted transport plans of the form $J(\pi)$ are of Monge type in the sense that they are concentrated on the graph of a function. Hence, 
${\overline{{\rm WT}}}_C(\mu, \nu)$ can be considered as a relaxation of ${\rm WT}_C(\mu, \nu)$ in the same way as the Kantorovich formulation relaxes the Monge formulation of the transport problem. In particular, it was established in \cite[Lemma~2.1]{BaBePa18} that in the standard setting, where $C(x,\cdot)$ is convex, the value of the weak transport problem and its relaxation coincide, i.e.
\begin{align}\label{eq:Relax_Eq}
    \overline{\rm WT}_C(\mu,\nu) = {\rm WT}_C(\mu,\nu). 
\end{align}

The relaxed WOT problem is of particular interest in the setting of non-convex cost introduced in \Cref{ssec:noconvexity}. A major reason for this is that primal attainment is guaranteed for $\overline{\rm WT}_C(\mu,\nu)$, while failing in general for ${\rm WT}_C(\mu,\nu)$ in the non-convex case. Further note that the inequality \eqref{eq:RelaxIneq} is in general strict. 
However, it was shown in \cite[Proposition 3.8]{AcBePa20} that if $\rho \mapsto C(x,\rho)$ is lsc
\begin{equation}
    \label{eq:Cstarstar}
    \overline{\rm WT}_C(\mu,\nu) = {\rm WT}_{\overline{C}}(\mu,\nu),
\end{equation}
where $\overline{C}$ is defined such that $\overline{C}(x,\cdot)$ is the lsc convex hull of $C(x,\cdot)$ for every $x \in X$.
Furthermore, \cite[Theorem 3.9]{AcBePa20} states that if $\mu$ has no atoms and $C$ is continous with bounded $p$-growth, then 
\begin{align}\label{eq:Cstarstarcont}
    \overline{\rm WT}_C(\mu,\nu) =  \overline{\rm WT}_{\overline{C}}(\mu,\nu) = {\rm WT}_{C}(\mu,\nu).
\end{align}

We recall that the initial example of this section showed that duality for the non-relaxed WOT problem fails. However, it is true for its relaxed version, as was shown in \cite[Theorem~3.1]{BaBePa18}\footnote{We note that it was shown there for cost functions that are jointly lsc. The generalization of this proof to cost functions that are jointly Borel and lsc in the second argument follows line by line as in the proof of \Cref{thm:FTWOT_mainbody}.}. We obtain the following variant of the fundamental theorem of WOT.

\begin{theorem}\label{thm:FTWOT_mainbody_L}
    Let $\mu \in \cP(X)$, $\nu \in \cP_p(Y)$ and let $C:X\times \P_p(Y)\to \R \cup \{+\infty\}$ be measurable and $\rho\mapsto C(x,\rho)$
 lsc for the $p$-weak convergence on $\P_p(Y)$. 
    Suppose that there are $a_\ell \in \mathcal L^1(\mu), b_\ell \in \mathcal L^1(\nu)$ such that $C(x,\rho) \ge a_\ell(x) + \rho(b_\ell)$, then:
    \begin{enumerate}[\upshape(i)]
    \item\label{it:FTWOT_main1_L} (primal attainment) ${\overline{{\rm WT}}}_C(\mu,\nu)$ is attained,
    \item\label{it:FTWOT_main2_L} (duality) we have ${\overline{{\rm WT}}}_C(\mu,\nu) = \DW{C}{\mu}{\nu}$ where 
        \begin{align}\label{eq:WOTdual_L}
            \DW{C}{\mu}{\nu}=
            \sup\{\mu(\ppsi^C)+\nu(\ppsi): \ppsi \in \mathcal L^1(\nu) \text{ s.t.\ } \ppsi^C \in \mathcal L^1(\mu) \}.
        \end{align}
    \item\label{it:FTWOT_main3_L} (dual attainment) If $C$ satisfies conditions \eqref{eq:WOTbound} and \eqref{eq:WOTcont},
    then $D_C(\mu,\nu)$ is attained.
    \item\label{it:FTWOT_main4_L} (complementary slackness) If ${\overline{{\rm WT}}}_C(\mu,\nu) < \infty$, a pair of candidates $(\pi, (\pphi, \ppsi))$ is optimal if and only if
    \begin{align}\label{eq:WOTCS_L}
        C(x,\rho)= \pphi(x)+\rho(\ppsi), \quad P(dx,d\rho)\text{-a.s.}
    \end{align}
    In this case, $\pphi(x) = \ppsi^C(x)$ $\mu$-a.s.
\end{enumerate}
\end{theorem}

\begin{remark}
We note that already in the setting of the non-relaxed WOT with convex cost, the duality \eqref{eq:WOTdual_L} is used to establish the WOT duality (on general Polish spaces). This is because one equips the set of couplings $\cpl(\mu, \nu)$ with the adapted weak topology (see e.g.\ \cite{BaBaBeEd19b}) which is stronger than the usual weak topology but not compact. Both the set of filtered processes and the set $\Lambda(\mu, \nu)$ are compactifications of $\cpl(\mu, \nu)$. It follows from \cite{BaBePa21} that the two viewpoints are equivalent.\hfill$\diamond$
\end{remark}

\section{Barycentric costs}\label{sec:barycentric_costs}
A particularly relevant class of cost functions consists of only depending on the barycenter of the measure $\rho$. More specifically, we consider the case $X = Y = \R^d$ with cost functions $C: X \times \cP_1(Y) \to \R$ of the form
\begin{align*}
     C(x,\rho) = \vartheta(x-\mean{\rho}),
\end{align*}
where $\vartheta : \R^d \to \R$ is convex and $\mean{\rho} = \int y \, \rho(dy)$. Then the weak optimal transport problem with cost $C$ reads as
\begin{align}\label{eq:barycentric_primal}
    {\rm BT}_\vartheta(\mu,\nu) = \min_{\pi \in \cpl(\mu,\nu) } \int \vartheta(x - \mean{\pi_x}) \, \mu(dx).
\end{align}
Recall that two probability measures $\eta, \rho \in \cP_1(\R^d)$ are said to be in convex order, denoted by $\eta \le_c \rho$ if $\int f \,d\eta \le \int f \, d\rho$ for every convex function $f : \R^d \to \R$. It is easy to observe that \eqref{eq:barycentric_primal} is closely related to the problem of projecting $\mu$ onto $\{\eta : \eta \le_c \nu \}$ w.r.t.\ the value of the classical transport problem with cost $\vartheta(x-y)$, i.e.\
\begin{align}\label{eq:proj_problem}
    {\rm BT}_\vartheta(\mu,\nu) = \min_{\eta \le_{c} \nu}  {\rm T}_\vartheta(\mu,\eta). 
    \end{align}
Specifically, if $\pi$ is an optimizer for ${\rm BT}_\vartheta(\mu,\nu)$, then $\eta := (x \mapsto \mean{\pi_x})_\# \mu$ is optimal in the minimum on the right-hand side of \eqref{eq:proj_problem} and $\xi := (x \mapsto (x, \mean{\pi_x}))_\#\mu \in \cpl(\mu,\eta)$ is an optimal coupling between $\mu$ and $\eta$ for the cost $c(x,y)=\vartheta(x-y)$. Conversely, if $\eta \le_c \nu$ and $\xi \in \cpl(\mu,\eta)$ are both optimal and $\kappa$ is any martingale coupling from $\eta$ to $\nu$, then the conditionally independent gluing of $\xi$ and $\kappa$ is optimal for ${\rm BT}_\vartheta(\mu,\nu)$. Note that this argument also shows that if $\eta$ is optimal in \eqref{eq:proj_problem}, then there is an optimal Monge coupling between $\mu$ and $\eta$ (without any regularity assumption on $\mu$). 

In contrast to the previous sections, we change the sign convention of the dual potential according to the rule $(\phi,\psi):=(f,-g)$. This is because $g$ turns out to be concave in the setting of this section. As we will often use techniques from convex analysis, it significantly simplifies the notation to work with the convex function $\psi = -g$.
We state the main result of this section.

\begin{theorem}\label{thm:StGaMc}
Let $\vartheta: \R^d \to \R$ be convex and assume that there exist $a \in L^1(\mu)$ and $b \in L^1(\nu)$ such that $\vartheta(x-y) \le a(x)+b(y)$ for all $x,y \in \R^d$.
\begin{enumerate}[\upshape(i)]
    \item\label{it:StGaMc1} The dual problem of \eqref{eq:barycentric_primal} is given by 
    \begin{align}\label{eq:baryDual}
    \sup_{\psi \text{ \rm convex, lsc}}   \mu (\vartheta \Box \psi) - \nu(\psi).
    \end{align}
    We have strong duality, primal attainment and dual attainment. 
    \item\label{it:StGaMc2} We have the following complementary slackness criterion: 
    The coupling $\pi \in \cpl(\mu,\nu)$ and the convex function $\psi \in L^1(\nu)$ are both optimal if and only if for $\mu$-a.e.\ $x$
    \[
    \vartheta \Box \psi (x)= \vartheta(x- \mean{\pi_x}) + \pi_x(\psi).
    \]

     \item\label{it:StGaMc3} If $\pi \in \cpl(\mu,\nu)$ is optimal for ${\rm BT}_\vartheta(\mu,\nu)$ and $\psi$ is a dual optimizer, writing $\phi := \vartheta \Box \psi$, we have $\mean{\pi_x} \in \partial^\vartheta \phi(x)$ for $\mu$-a.s.

    \item\label{it:StGaMc4} Suppose that $\vartheta$ is strictly convex. If $\pi, \tilde{\pi} \in \cpl(\mu,\nu)$ are both optimal for ${\rm BT}_\vartheta(\mu,\nu)$, then $\mean{\pi_x} = \mean{\tilde{\pi}_x}$ $\mu$-a.s. Hence, the optimal $\eta \le_c \nu$ and the optimal coupling $\gamma \in \cpl(\mu,\eta)$ mentioned in \eqref{eq:proj_problem} are both unique. 

    \item\label{it:StGaMc5} Suppose that $\vartheta$ is strictly convex, differentiable and supercoercive. If $\psi$ and $\pi$ are optimal, then $\phi := \vartheta \Box \psi \in C^1(\R^d)$ and $\mean{\pi_x}= x- \nabla \vartheta^\ast(\nabla(\phi(x)))$ $\mu$-a.s.
\end{enumerate}
\end{theorem}
Note that the growth condition on $\vartheta$ is automatically satisfied if $\mu, \nu$ have bounded support, or more generally, if $\mu, \nu$ have finite $p$-th moments and $\vartheta$ does not grow faster than $|x|^p$.
By abuse of notation, we write $\partial^\vartheta \phi(x)$ to denote the $c_\vartheta$-subdifferential of $\phi$ at $x$, where $c_\vartheta(x,y) = \vartheta(x-y)$, that is, the set
$\partial^\vartheta \phi(x) = \{ y \in \R^d : \phi(x) + \vartheta(x-y) \le \phi(z) + \vartheta(z-y), \, \forall z \in \R^d \}$.

\subsection{Applications of \Cref{thm:StGaMc}} 
Note that Theorem~\ref{thm:StGaMc} is an extension of the Brenier--Strassen theorem in a similar way as the Gangbo--McCann theorem \cite[Theorem~1.2]{GaMc96} is an extension of Brenier's theorem. We note that \eqref{eq:proj_problem} was already established by Gozlan and Julliet \cite{GoJu18} for general convex $\vartheta$; moreover they establish the existence of dual maximizers in the case of compactly supported measures.

\begin{remark}\label{rem:StGaMcStr}
We note that applying \Cref{thm:StGaMc}\eqref{it:StGaMc1} to a convex function $\vartheta: \R^d \to \R$ which attains its unique minimum in zero, implies Strassen's theorem \cite{St65}. 
To see this assume wlog that $\vartheta \ge 0$ and $\vartheta(0)=0$.
As $\vartheta\Box \psi \le \psi$, we find 
\begin{equation}
    \label{eq:rem.StGaMcStr}
    0 \le {\rm BT}_\vartheta(\mu,\nu) = \sup_{\psi \text{ \rm convex, lsc}}   \mu( \vartheta \Box \psi)  - \nu( \psi)  \le \sup_{\psi \text{ \rm convex, lsc}} \mu(\psi)  - \nu(\psi).
\end{equation}
Suppose that the measures $\mu,\nu \in \cP_1(\R^d)$ are in convex order, then we have by \eqref{eq:rem.StGaMcStr} that ${\rm BT}_\vartheta(\mu,\nu) = 0$.
Hence, there is $\pi \in \cpl(\mu,\nu)$ such that $\int \vartheta(x-\mean{\pi_x}) \, \mu(dx)=0$. As $\vartheta$ has its unique minimum in 0, we have $\mean{\pi_x}=x$ $\mu$-a.s., i.e.\ $\pi$ is a martingale coupling. \hfill$\diamond$
\end{remark}

\begin{example}[$\vartheta(x)=|x|^p$; $p$-Wasserstein projection in convex order] 
A particularly relevant special case of Theorem~\ref{thm:StGaMc} is $\vartheta(x)=|x|^p$. 
Then, for $\mu,\nu \in \cP_p(\R^d)$, the problem \eqref{eq:proj_problem} reads as 
\begin{align}\label{eq:WpProj}
\inf \{ \W_p^p(\mu,\eta) : \eta \le_c \nu \},
\end{align}
which is precisely finding the metric projection of $\mu$ onto the set $\{ \eta \in \cP_p(\R^d) : \eta \le_c \nu \}$ in the $p$-Wasserstein space $(\cP_p(\R^d),\W_p)$. Theorem~\ref{thm:StGaMc} states existence, uniqueness (if $p>1$) of solutions to this projection problem and gives a description of this solution. \hfill $\diamond$ 
\end{example}

\begin{example}[$\vartheta(x)= \frac{1}{2}|x|^2$; Brenier--Strassen theorem]
In this case, we recover precisely the Brenier--Strassen theorem of Gozlan and Julliet \cite{GoJu18}.

By \Cref{thm:StGaMc}, there exist a primal optimizer $\pi \in \cpl(\mu,\nu)$ and a dual optimizer $\psi$. Then the function $\phi := \frac{1}{2} | \cdot |^2 \Box \psi$ in the dual problem \eqref{eq:baryDual} is the Moreau envelope of $\psi$ (with parameter 1). 
Moreover, there is a unique $\eta \in \cP_2(\R^d)$ which solves the projection problem \eqref{eq:WpProj} for $p=2$ and the optimal coupling in $\W_2^2(\mu,\eta)$ is unique and induced by the Monge map 
$
T(x)= \mean{\pi_x}.
$
By the complementary slackness condition (\Cref{thm:StGaMc}\eqref{it:StGaMc2}), we find
\begin{align}\label{eq:ProxPsi}
    \frac{1}{2}|\cdot|^2 \Box \psi(x) \le  \frac{1}{2}|x - T(x)|^2 + \psi(T(x)) \le \frac{1}{2}|x -T(x)|^2 + \pi_x(\psi) = \frac{1}{2}|\cdot|^2  \Box \psi(x).
\end{align}
Hence, all inequalities in \eqref{eq:ProxPsi} are in fact equalities. Note that \eqref{eq:ProxPsi} is then the defining equation of the proximal map of $\psi$ (see \Cref{sec:app1}), i.e.\ $T = \rm{Prox}_\psi$. It is well-known that the proximal map is differentiable with 1-Lipschitz gradient.\footnote{Note that $T=\rm{Prox}_\psi$ coincides with the transport map stated in \Cref{thm:StGaMc}\eqref{it:StGaMc5}. Indeed, \cite[Proposition~12.29]{BaCo11} asserts that ${\rm Prox}_\psi = {\rm Id} - \nabla( \frac{1}{2} | \cdot |^2 \Box \psi)$. As $\vartheta=\frac{1}{2}|\cdot|^2$, we have $\nabla \vartheta^\ast = {\rm Id}$ and hence ${\rm Prox}_\psi(x) = x- \nabla \vartheta^\ast(\nabla(\phi(x)))$. } \hfill $\diamond$   
\end{example}

Recall that a support function is a convex function $\vartheta: \R^d \to \R \cup \{+\infty \}$ that satisfies $\vartheta(tx)=t\vartheta(x)$ for every $x \in \R^d$ and $t \ge 0$. Note that support functions are subadditive, i.e., $\vartheta(x+y) \le \vartheta(x) + \vartheta(y)$.
In the case of support functions, the duality simplifies in the following way.

\begin{corollary}[Convex Kantorovich--Rubinstein for support functions]\label{cor:hSpt}
Let $\vartheta : \R^d \to \R$ be a support function. Then, 
\begin{align}
{\rm BT}_\vartheta(\mu,\nu)=& \max\left\{  \mu(\phi)  -   \nu(\phi)  : \phi \text{ \rm convex, } \phi(x_1) - \phi(x_2) \le \vartheta(x_1-x_2) \right\}. 
\end{align} 
\end{corollary}

Note that the condition $\phi(x_1) - \phi(x_2) \le \vartheta(x_1-x_2) \text{ for all } x_1,x_2 \in \R^d$ characterizes those convex functions $\phi$ such that there is a convex function $\psi$ with $\phi = \vartheta \Box \psi $, see \Cref{lem:RangeOfBox}. Hence, we could write the assertion also as
\begin{align}\label{eq:hSptBox}
{\rm BT}_\vartheta(\mu,\nu) 
=& \max\left\{  \mu( \vartheta\Box \psi)  - \nu( \vartheta\Box \psi)  : \psi \text{ \rm convex}  \right\}
\end{align}

\begin{proof}
As support functions have linear growth, the growth condition in \Cref{thm:StGaMc} is automatically satisfied, and we have duality with the problem \eqref{eq:baryDual} and attainment. 
By subadditivity of $\vartheta$, we have $\vartheta \Box \vartheta \Box \psi = \vartheta \Box \psi$, see \Cref{lem:RangeOfBox}.
As $\vartheta(0)=0$, we have $\vartheta\Box \psi \le \psi$. These two observations show that replacing a convex function $\psi$ with $\vartheta\Box \psi$ increases the value of the dual functional. Hence, we can restrict the maximum to functions $\phi$ of the form $\phi = \vartheta\Box \psi$. This implies the claim by again using that $\vartheta\Box \vartheta \Box \psi = \vartheta \Box \psi$.
\end{proof}

\begin{example}[$\vartheta(x)=\|x\|$; convex Kantorovich--Rubinstein] 
Let $\|\cdot\|$ be a norm on $\R^d$.
For $\mu, \nu \in \cP_1(\R^d)$, we have 
\begin{align}\label{eq:convKR}
    \min_{ \eta \le_c \nu } \W_1(\mu,\eta) = \min_{\pi \in \cpl(\mu,\nu)} \int \|x- \mean{\pi_x} \| \,\mu(dx) 
 = \max_{\substack{\phi\,  \rm{ convex,} \\ \text{$\|\cdot\|$-1-Lipschitz}}} \mu(\phi) - \nu(\phi).
\end{align}
This follows from \eqref{eq:proj_problem} and \Cref{cor:hSpt} because every norm is a support function and the condition $\phi(x_1) -\phi(x_2) \le \|x_1-x_2\|$ precisely states that $\phi$ is 1-Lipschitz w.r.t. $\| \cdot \|$. \hfill $\diamond$ 
\end{example}

\begin{example}[$\vartheta(x)=x_+$; increasing convex Kantorovich--Rubinstein]
Let $\mu, \nu \in \cP_1(\R)$. Then, we have 
\[
\min_{\pi \in \cpl(\mu,\nu)} \int (x- \mean{\pi_x})_+ \,\mu(dx) 
 = \max_{\substack{\phi\,  \rm{increasing \, convex, } \\ \rm{1\text{-Lipschitz}}}} \mu(\phi) - \nu(\phi).
\]
This follows from \Cref{cor:hSpt} noting that the condition $\phi(x_1) -\phi(x_2) \le (x_1-x_2)_+$ precisely states that $\phi$ is 1-Lipschitz and increasing. \hfill $\diamond$
\end{example}

\begin{example}[Multidimensional increasing convex Kantorovich--Rubinstein]
Let $\le$ be the coordinatewise order on $\R^d$, i.e. $x \le y$ if $x_i \le y_i$ for every $i \in \{1, \dots, d\}$. Then, we have
\[
\min_{\pi \in \cpl(\mu,\nu)} \int \max_{i=1,\dots,d}(x_i-\mean{\pi_x}_i)_+  \, \mu(dx) = \max_{\substack{\phi\,\textrm{convex, $\| \cdot \|_1$-1-Lipschitz} \\ \textrm{increasing w.r.t.\ $\le$}}} \mu(\phi) - \nu(\phi).
\]

More generally, let $\le$ be a partial order on $\R^d$ that is compatible with the linear structure and continuous (i.e.\ if $x_n \to x$ and $x_n \ge0$, then $x \ge 0$) and let $\| \cdot \|$ be a norm on $\R^d$. Then, 
\[
\min_{\pi \in \cpl(\mu,\nu)} \int {\rm{dist}}_{\| \cdot \|}(x-\mean{\pi_x}, -P)  \, \mu(dx) = \max_{\substack{\phi\,  \text{convex, $\|\cdot\|$-1-Lipschitz,} \\ \text{increasing\, w.r.t.\ $\le$}}}\mu(\phi) - \nu(\phi) ,
\] 
where $P := \{ y \in \R^d : y \ge 0\}$ is the positive cone for the order $\le$ and ${\rm dist}_{\|\cdot\|}(\,\cdot\, , -P)$ denotes the distance to $-P$ w.r.t. the norm $\| \cdot \|$. 

To see this, we apply \eqref{eq:hSptBox} with the function $\vartheta(x) = {\rm{dist}}_{\| \cdot \|}(x, -P)$. Then it remains to observe that a convex function $\phi$ is $\|\cdot\|$-1-Lipschitz and $\le$-increasing if and only if it is of the form $\phi=\vartheta \Box \psi$ for some convex function $\psi$. For this, note that 
$$
\vartheta(x) = {\rm dist}_{\|\cdot\|}(x, -P) = \inf_{y \in -P} \|x-y\| = \|\cdot\| \Box \chi_{(-P)} (x). 
$$
Using this, \Cref{cor:RangeBoxIntersec} implies that $\phi = \vartheta \Box \psi$ for some convex function $\psi$ if and only if there is a convex functions $\psi_1$ such that $\phi = \| \cdot \| \Box \psi_1$ and there is a convex functions $\psi_2$ such that $\phi = \chi_{(-P)} \Box \psi_2$, where $\chi$ denotes the convex indicator, see \Cref{sec:app2}.
The first condition is equivalent to $\phi$ being 1-Lipschitz w.r.t. $\|\cdot\|$ and the second condition is equivalent to $\phi$ being increasing w.r.t. $\le$. 
\hfill $\diamond$
\end{example}

\subsection{Mathematical finance interpretation of convex Kantorovich-Rubinstein}\label{ssec:Finance_interpretation}
 \Cref{cor:hSpt} admits a natural financial interpretation in terms of model-independent arbitrage under trading restrictions. 
To provide a concise formulation we recall that there is a one-to-one correspondence between support functions $\vartheta :\R^d \to \R$ and compact convex subsets of $\R^d$. The support function of a compact convex set $K \subset \R^d$ is given by
\begin{align}\label{eq:hK}
    \vartheta_K(x)=\sup_{y \in K} x \cdot y.
\end{align}
In this case $\partial \vartheta_K(0) =K$. As $\vartheta = \vartheta_{\partial \vartheta(0)}$ and $\partial \vartheta(0)$ is compact convex, every support function arises as in \eqref{eq:hK}.

We then have the following result which is essentially a reformulation of \Cref{cor:hSpt}.
\begin{corollary}\label{cor:hSptFinancy}
Let $K \subseteq \R^d$ be compact and convex. Then we have
\begin{align} \label{eq:SuperRep} \begin{split}
&\min_{\pi  \in \cpl(\mu,\nu)}  \int \vartheta_K(\mean{\pi_x}-x) \,\mu(dx) = \\
&\quad  \max\{w:\exists f_1, f_2, \Delta; \ w\leq (f_1(x_1)-\mu(f_1)) + (f_2(x_2)-\nu(f_2)) + \Delta(x_1)\cdot(x_2-x_1) \},
\end{split}
\end{align}
where $\Delta$ is measurable with values in $K$ and $f,g$ are Lipschitz. 
\end{corollary}
Corollary \ref{cor:hSptFinancy} can be seen as a robust quantitative FTAP (`fundmental theorem of asset pricing') for a two time-step market $X_1,X_2$ consisting of $d$ assets. In mathematical finance terms, $f_1,f_2$ are vanilla options with maturities $t=1$ and $t=2$, resp.\ which can be bought at prices $\mu(f_1)$ and $ \nu(f_2)$, resp.; this is based on the famous Breeden--Litzenberger observation \cite{BrLi78} that prices of liquidly traded call / put options can be expressed as distributions of the underlying $X$ under `pricing measures'. Then 
\[S_{f_1,f_2,\Delta}(x_1,x_2):=(f_1(x_1)-\mu(f_1)) + (f_2(x_2)-\nu(f_2)) + \Delta(x_1)\cdot(x_2-x_1) \]
represents the gains / losses from self-financing trading, with static positions $f_1(X_1), f_2(X_2)$ and  holding $\Delta(X_1)$ assets from time $t=1$ to time $t=2$.
The right hand side of \eqref{eq:SuperRep} is the maximal risk less profit (``arbitrage'') an investor can achieve by trading under the restriction $\Delta(X_1)\in K$.

To illustrate Corollary \ref{cor:hSptFinancy}, we consider first the case $d=1$ and $K=[-1,1]$ where $\vartheta_K(x)=|x|$. If the market satisfies the strict No Arbitrage condition, the maximal risk less profit equals $0$ and \eqref{eq:SuperRep} recovers the existence of a martingale measure $\pi$ consistent with $\mu, \nu$. More generally, if the maximal risk less profit under the trading restriction $|\Delta(X_1)| \leq 1$ has level $\ell$, \eqref{eq:SuperRep} yields the existence of a consistent `$\ell$-almost' martingale pricing measure.

To consider another example, let again $d=1$ and $K=[0,1]$ such that $\vartheta_K(x)=(x)_+$. Financially, $\Delta(x)\geq 0$ means that no short-selling is allowed. Corollary \ref{cor:hSptFinancy} then connects the maximal level of arbitrage per stock to the existence of an `almost' super-martingale measure.

\begin{proof}[Proof of \Cref{cor:hSptFinancy}]
First note that the right hand side of \eqref{eq:SuperRep} is equal to $\max \mu(f_1) + \nu(f_2)$, where the maximum is taken of all $(f_1,f_2,\Delta)$ such that $\Delta(x_1)\in K$ and $f_1(x_1)+f_2(x_2)+\Delta(x_1)\cdot (x_2-x_1)\ge 0$. By \Cref{thm:StGaMc}\eqref{it:StGaMc1} and \eqref{eq:hSptBox} applied with $\vartheta:= \vartheta_{-K}$, it suffices to show that 
\[
\max_{\zeta \text{ convex}} \mu(\vartheta \Box \zeta) - \nu(\vartheta \Box \zeta) \leq 
\!\! \max_{\substack{(f_1, f_2, \Delta) \text{ s.t. } \Delta(x_1) \in K, \\ f_1(x_1) + f_2(x_2) + \Delta(x_1) \cdot (x_2 - x_1) \geq 0 \\ }} \!\!\!\!\!\!\! \mu(f_1) + \nu(f_2) \leq 
\max_{ \substack{ \phi(x_1) - \rho(\psi) \le \\ \vartheta(x-\mean{\rho}) } } \mu(\phi) - \nu(\psi).
\]

To see the first inequality, let $\zeta$ be optimal and define $f_1 := -(\vartheta\Box \zeta)$, $f_2:=  \vartheta \Box \zeta$ and let $\Delta$ be a selector of $-\partial( \vartheta \Box \zeta)$. Then the sub-differential inequality $\vartheta \Box \zeta (x_2) \ge \vartheta \Box \zeta(x_1) + H(x_1)\cdot(x_2-x_1)$ yields that $f_1(x_1) + f_2(x_2) + \Delta(x_1) \cdot (x_2 - x_1) \geq 0$. As $\partial (\vartheta\Box \zeta)(x_1)\subseteq -K$ (see \Cref{lem:RangePartialBox}), we have $\Delta(x_1) \in K$. 

To see the second inequality, let $(f_1,f_2,\Delta)$ be optimal and set $\phi := -f_1$, $\psi := f_2$. Then, integrating the admissibility condition for $(f_1,f_2,\Delta)$ for fixed $x_1$ w.r.t.\ $\rho(dx_2)$ yields
\[
\phi(x_1) \le -\Delta(x_1) \cdot (x_1 -\mean{\rho}) + \rho(\psi) \le \vartheta(x_1-\mean{\rho}) + \rho(\psi),
\]
where the second inequality is the subdifferential inequality for $\vartheta$ in the point 0, noting that $-\Delta(x_1) \in -K = \partial \vartheta(0)$.  
\end{proof}

\subsection{Proof of \Cref{thm:StGaMc}} We aim to prove \Cref{thm:StGaMc} by invoking the fundamental theorem of weak optimal transport. As the first step, we calculate the $C$-transform.

\begin{lemma}\label{lem:C_trafo_with_convex_h}
Let $\vartheta: \R^d \to \R$ be convex, $C(x,\rho)=\vartheta(x-\mean{\rho})$ and let $\psi : \R^d \to \R  \cup \{+\infty\}$ be proper. Then, we have
\[
(-\psi)^C = \vartheta \Box \psi^{\ast\ast}. 
\]
If $\vartheta$ is supercoercive, we have $\psi^C(x) > - \infty$ for every $x \in \R^d$.
\end{lemma}

\begin{proof}
Using the definition of the $C$-transform, we find
\begin{align*}
    (-\psi)^C(x) = \inf_{\rho \in \cP(Y)} C(x,\rho) +\rho(\psi) = \inf_{y \in \R^d} \vartheta(x-y) + \inf_{\mean{\rho}=y} \int \psi \, d\rho = \vartheta \Box \widehat{\psi}(x), 
\end{align*}
where $\widehat{\psi}(y) = \inf_{\mean{\rho}=y} \rho(\psi)$ is the convex envelope of $\psi$. 
As $\vartheta$ is a finite convex function, we have $\vartheta \Box \widehat{\psi} = \vartheta \Box \psi^{\ast\ast}$, see \eqref{eq:BoxLscHull} in \Cref{sec:app1}.
\end{proof}

\begin{proof}[Proof of \Cref{thm:StGaMc}\eqref{it:StGaMc1} and \eqref{it:StGaMc2}]
We first check that $C$ satisfies the assumptions of the fundamental theorem of weak optimal transport (\Cref{thm:FTWOT_mainbody}). First note that $C : \R^d \times \cP_1(\R^d) \to \R$ is continuous because $\vartheta$ is continuous and $\cP_1(\R^d) \to \R^d: \rho \mapsto \mean{\rho}$ is continuous. The convexity of $C$ in the second argument follows from the convexity of $\vartheta$. Moreover, $C$ satisfies the growth condition because
\[
C(x,\rho) = \vartheta(x-\mean{\rho}) \le \int \vartheta(x-y) \, \rho(dy)
\le a(x) + \rho(b).
\]
Since $\vartheta$ is convex and thus lower bounded by an affine function, $C$ satisfies the lower bound, i.e., $\vartheta(z) \ge r + v\cdot z$ for some $r \in \R$ and $v \in \R^d$, showing that $C(x,\rho) \ge r - |v\cdot x| - \int |x\cdot v| \, \rho(dx)$. 

Note that $(-\psi)^C= \vartheta\Box\psi^{\ast\ast}$ by \Cref{lem:C_trafo_with_convex_h}. Hence, \Cref{thm:FTWOT_mainbody} yields primal and dual existence and strong duality for the dual problem
\begin{align}\label{eq:baryDProof}
    \sup \left\{ \mu(\vartheta\Box\psi^{\ast\ast}) - \nu(\psi) ; \psi \in \cL^1(\nu) \right\}.
\end{align}
As $(-\psi)^C = (-\psi^{\ast\ast})^C$ and $\psi \ge \psi^{\ast\ast}$, we can restrict the supremum in \eqref{eq:baryDProof} to lsc convex functions. 

Note that the complementary slackness condition, cf.\ \Cref{thm:StGaMc} \eqref{it:StGaMc2}, applied to a pair of the form $((-\psi)^C,\psi)$ yields \eqref{it:StGaMc2}.
\end{proof}

\begin{proof}[Proof of \Cref{thm:StGaMc}\eqref{it:StGaMc3}] 
Let $\psi$ and $\pi$ be optimal. 
Using the definition of infimal convolution, the convexity of $\psi$, and complementary slackness (\Cref{thm:StGaMc} \eqref{it:StGaMc2}), we get
\[
\vartheta \Box \psi(x) \le  \vartheta(x - \mean{\pi_x}) + \psi(\mean{\pi_x}) \le \vartheta(x - \mean{\pi_x}) + \pi_x(\psi) = \vartheta \Box \psi(x).
\]
Hence, all of these inequalities are, in fact, equalities, i.e., $\mean{\pi_x}$ is a minimizer in the definition of $\vartheta \Box \psi(x)$. 
It is easy to see that this is equivalent to $\mean{\pi_x} \in \partial^c \phi(x)$, where we write $\phi = \vartheta \Box \psi$ and $c(x,y)=\vartheta(x-y)$. To see this, recall that
\[
\partial^c\phi (x) = \{y \in \R^d : \phi(v) \le \phi(x) +c(v,y) - c(x,y) \text{ for all } v \in \R^d \},
\]
and note that, for every $v \in \R^d$, we have 
\begin{align*}    
    \vartheta\Box \psi (v) &\le \vartheta(v-\mean{\pi_x}) +\psi(\mean{\pi_x}) 
    \\
    &= \vartheta\Box\psi(x) + \vartheta(v-\mean{\pi_x}) - \vartheta(x-\mean{\pi_x}). \qedhere
\end{align*}
\end{proof}

For the proof of \Cref{thm:StGaMc}\eqref{it:StGaMc4} we need the following observation.
\begin{lemma}\label{lem:optMonge}
Suppose that $\vartheta : \R^d \to \R$ is strictly convex, $\eta$ is optimal in $\min_{\eta \le_{c} \nu} \OT{c}{\mu}{\eta}$, and $\xi \in \cpl(\mu, \eta)$ is optimal. 
Then $\xi$ is of Monge type.
\end{lemma}
\begin{proof}
Write $T(x) = \mean{\xi_x}$.
We define the measure $\tilde \eta := T_\# \mu $ and the coupling $\tilde \xi := (\id,T)_\# \mu \in \cpl(\mu,\tilde\eta)$. 
Using Jensen's inequality we find for every convex function $f : \R^d \to \R$
\[
\int f \,d\tilde \eta = \int f(\mean{\xi_x}) \, \mu(dx) \le \iint f(y) \, \xi_x(dy)\,\mu(dx) = \int f \, d\eta.
\]
Hence, $\tilde \eta \le_c \eta \le_c \nu$, so $\tilde \eta$ is admissible for the problem $\min_{\eta \le_{c} \nu} \OT{c}{\mu}{\eta}$. As $\eta$ is optimal for this problem and $\xi \in \cpl(\mu,\eta)$ is optimal, we get using Jensen's inequality
\[
\int \vartheta(x-y) \,\xi_x(dy)\,\mu(dx) \le \OT{c}{\mu}{\tilde \eta} \le \int \vartheta(x-T(x)) \, \mu(dx) \le \iint \vartheta(x-y) \, \xi_x(dy)\, \mu(dx).
\]
As $\vartheta$ is strictly convex, this yields that for $\mu$-a.e.\ $x$ we have $y=T(x)$ $\xi_x$-a.s. Hence, $\xi=(\id,T)_\#\mu$. 
\end{proof}

\begin{proof}[Proof of \Cref{thm:StGaMc}\eqref{it:StGaMc4}] 
Let $\pi^1, \pi^2 \in \cpl(\mu,\nu)$ be optimal and let $\psi$ be a dual optimizer. By the consideration in the proof of claim \eqref{it:StGaMc3}, both $\mean{\pi^1_x}$ and $\mean{\pi^2_x}$ are minimizers of $\vartheta \Box \psi(x)$. 
As $\vartheta$ is strictly convex, this minimizer is unique, so $\mean{\pi^1_x} = \mean{\pi^2_x}$.

To show the second uniqueness claim, let $\eta^i \le_c \nu$ and $\xi^i \in \cpl(\mu,\eta^i)$ be optimal for $i \in \{1,2\}$. By \Cref{lem:optMonge}, there are maps $T^i$ such that $\xi^i = (\id,T^i)_\ast \mu$ and by Strassen's theorem (see also \Cref{rem:StGaMcStr}) there are martingale couplings $\kappa^i$ from $\eta^i$ to $\nu$. Then, the couplings $\pi^i(dx,dz) := \kappa^i_{T^i(x)}(dz)\mu(dx)$ are optimal for $T_\vartheta(\mu|\nu)$. By the previous considerations, we find 
$
T^1(x)= \mean{\pi^1_x} = \mean{\pi^2_x} = T^2(x)
$
$\mu$-a.s.\ and hence $\xi^1=\xi^2$. In particular, their second marginals $\eta^1$ and $\eta^2$ have to coincide. 
\end{proof}

\begin{proof}[Proof of \Cref{thm:StGaMc}\eqref{it:StGaMc5}] 
We write $\phi := \vartheta \Box \psi$. Note that $\phi \in C(\R^d)$ because it is the infimal convolution of a lsc convex function with a differentiable super-coercive convex function, see e.g.\ \cite[Corollary~18.8]{BaCo11} and \cite[Theorem~2.2.2]{BoVa10}.  Using the differentiability of $\vartheta$ and $\psi$, we find (see e.g.\ \cite[Lemma~3.1]{GaMc96}) 
 \[
     y \in \partial^c \phi(x) \iff \nabla \phi(x) = \nabla \vartheta(x-y) \iff x-y = \nabla \vartheta^\ast( \nabla \phi (x)). 
 \]
For the last equivalence note that $\vartheta^\ast$ is differentiable as $\vartheta$ is strictly convex, cf.\ \Cref{sec:app2}. 
\end{proof}

\section{Regularized optimal transport} \label{sec:reg_OT} 
The aim of this section is to outline how (entropic) regularized optimal transport is covered by weak optimal transport and to derive the structure results for these problems from the fundamental theorem of weak optimal transport.
\subsection{Entropic optimal transport}
The entropic transport problem is given by
\begin{align}\label{eq:EOTWOT}
    \EOT{c}{\eps}{\mu}{\nu}= \inf_{\pi \in \cpl(\mu,\nu)} \int c\, d\pi + \eps H ( \pi | \mu \otimes \nu), 
\end{align}
where $\eps >0$, and $H$ denotes relative entropy of $\pi$ w.r.t.\ the product measure $\mu \otimes \nu$, i.e.
\[
H(\pi | \mu \otimes \nu ) = \int \log\left( \frac{d\pi}{d\mu\otimes\nu} \right) \,d\pi = \int \frac{d\pi}{d\mu\otimes\nu} \log\left( \frac{d\pi}{d\mu\otimes\nu} \right) \,d\mu\otimes \nu.
\]
Further, we assume that the cost function $c: X \times Y \to \R$ is Borel, lower bounded and there exist $a \in \cL^1(\mu), b \in \cL^1(\nu)$ such that $c(x,y) \le a(x) +b(y)$. Problem \eqref{eq:EOTWOT} can be regarded as a weak transport problem with the weak cost $C: X \times \cP(Y) \to \R \cup \{+\infty \}$ defined as 
\begin{align}\label{eq:EOTWOTcost}
C(x, \rho) = \int c(x,\cdot) \, d\rho + \eps \int \frac{d\rho}{d\nu} \log\left( \frac{d\rho}{d\nu} \right) \, d\nu,
\end{align}
and $+\infty$ if $\rho$ is not absolutely continuous w.r.t.\ $\nu$.
Applying the fundamental theorem of weak optimal transport guarantees existence of dual optimizers and a complementary slackness criterion for optimality.
\begin{proposition}\label{prop:EOTWOTdual}
The weak optimal transport dual problem for \eqref{eq:EOTWOT} is given by
\begin{align}\label{eq:EOTWOTdual}
    \sup \left\{\mu(\pphi) +  \nu(\ppsi)  \,\middle|\, (\pphi,\ppsi) \in \cL^1(\mu) \times \cL^1(\nu) : \pphi(x) + \rho(\ppsi) \le \int c(x, \cdot)\, d\rho + \eps H(\rho|\nu) \right\}.
\end{align}
We have strong duality, primal attainment and dual attainment. Moreover, for $\pi \in \cpl(\mu,\nu)$ and an admissible pair $(\pphi,\ppsi)$ the following are equivalent:
\begin{enumerate}[\upshape(i)]
    \item $\pi$ is a primal optimizer and $(\pphi,\ppsi)$ is a dual optimizer
    \item We have for $\mu$-a.e. $x$
    \begin{align}\label{eq:EOT-WOT-slackness}
    \pi_x(c(x, \cdot))  + \eps H(\pi_x|\nu) = \pphi(x) + \pi_x(g). 
\end{align}
\end{enumerate}
\end{proposition}

\begin{proof}
To simplify the notation in the proof, we set $\eps=1$.
We need to check that the cost function \eqref{eq:EOTWOTcost} satisfies the assumptions of the fundamental theorem of weak optimal transport (\Cref{thm:FTWOT_mainbody}) which then yields the claim. $C$ satisfies the growth condition \eqref{eq:WOTbound} with the super-coercive increasing convex function $h(t) = t \log(t)_+ +1$. It is clear that $C$ is Borel. 

To see that $\rho \mapsto C(x, \rho)$ is convex and lsc w.r.t.\ the weak topology, we write it as a relative entropy and use that this is known to be convex and lsc (see e.g.\ \cite[Lemma 1.3]{Nu22}). Indeed, writing $\alpha_x := \int e^{-c(x,\cdot)} d\nu$ and $\nu^x := \alpha_x^{-1} e^{-c(x,\cdot)} \nu$, we find that $C(x,\rho) = H(\rho|\nu^x)-\log(\alpha_x)$.

To check that $C$ satisfies condition \eqref{eq:WOTcont}, let $(Y_n)_n$ be an increasing sequence of Borel subsets of $Y$ satisfying $\bigcup_n Y_n =Y$. Writing $\rho_n = \frac{1}{\rho(Y_n)}\rho|_{Y_n}$, we have for every $x \in X$ and $\rho \in \cP(Y)$ with $\rho(Y_n)\to 1$ that $\frac{d\rho_n}{d\nu} \to \frac{d\rho}{d\nu}$ $\nu$-a.s. As $ \frac{d\rho_n}{d\nu} \le \frac{1}{\rho(Y_1)} \frac{d\rho}{d\nu}$ for every $n \in \N$ and as $c$ and $t \mapsto t \log(t)$ are lower bounded, dominated convergence yields
\begin{align*}
C(x,\rho) = \int c(x,\cdot) \frac{d\rho}{d\nu} + \frac{d\rho}{d\nu} \log\Big( \frac{d\rho}{d\nu} \Big)\, d\nu = \lim_n  \int c(x,\cdot) \frac{d\rho_n}{d\nu} + \frac{d\rho_n}{d\nu} \log\Big( \frac{d\rho_n}{d\nu} \Big)\, d\nu  = \lim_n C(x,\rho_n), 
\end{align*}
showing that $C$ satisfies the continuity assumption \eqref{eq:WOTcont}. 
\end{proof}
Next, we derive the classical structure theorem of entropic optimal transport (see e.g.\ \cite[Theorem~2.9]{Le14}, \cite[Theorem~4.2]{Nu22}) from the complementary slackness condition \eqref{eq:EOT-WOT-slackness}. 
\begin{theorem}\label{prop:EOT_opt_density}
Let $\pi \in \cpl(\mu,\nu)$. Then $\pi$ is optimal for $\EOT{c}{\eps}{\mu}{\nu}$ if and only if there exist measurable $f : X \to \R$, $g : Y \to \R$ such that
\begin{align}\label{eq:proddensity}
    \frac{d\pi}{d\mu \otimes \nu} = \exp\left(\frac{-c+f \oplus g}{\eps}  \right). 
\end{align}
 In this case $(\pphi,\ppsi)$ is optimal in the dual problem \eqref{eq:EOTWOTdual}.
\end{theorem}
\begin{proof}
By replacing $C$ with $\frac{1}{\eps}C$ and then scaling the respective dual potentials by $\eps$, we assume wlog that $\eps=1$. We first show that if $\pi \in \cpl(\mu,\nu)$ is optimal and $(\pphi,\ppsi)$ is optimal, then \eqref{eq:proddensity} holds. By possibly changing the optimizers on nullsets, we assume wlog that \eqref{eq:EOT-WOT-slackness} holds for every $x \in X$. We fix $x \in X$ and aim to show that 
\begin{align}\label{eq:prfFTEOT0}
    \frac{d\pi_x}{d\nu} = \exp(\pphi(x)+\ppsi-c(x,\cdot)).
\end{align}
To that end, let $\eta \in \cP(Y)$ and set $\rho_t := (1-t)\pi_x + t\eta$. As $\rho_t$ is a probability measure for every $t \in [0,1]$, the admissibility condition for $(\pphi,\ppsi)$ yields
\[
\pphi(x) + \int \ppsi \,d\rho_t \le \int c(x,\cdot)  \, d\rho_t + H(\rho_t|\nu).
\]
Subtracting \eqref{eq:EOT-WOT-slackness} from this yields the following inequality (which is an equality for $t=0$ because $\rho_0=\pi_x$)
\[
0 \le \int c(x,\cdot)-\ppsi \,d(\rho_t-\pi_x) + \int h\Big( \frac{d\rho_t}{d\nu} \Big) - h\Big( \frac{d\pi_x}{d\nu} \Big) \,  d\nu,
\]
where $h(t)= t\log(t) -t$. Note that $h'(t)=\log(t)$. Taking the right derivative at $t=0$ yields 
\begin{align}\label{eq:prfFTEOT1}
    0 &\le \frac{d}{dt}\Big|_{t=0^+} \int (c(x,\cdot)-\ppsi) \,d(\rho_t-\pi_x) + \int h\Big( \frac{d\rho_t}{d\nu} \Big) - h\Big( \frac{d\pi_x}{d\nu} \Big) \,  d\nu \notag\\
    &= \int (c(x,\cdot)-\ppsi) \,d(\eta -\pi_x) + \int \frac{d}{dt}\Big|_{t=0^+} h\Big( \frac{d\rho_t}{d\nu} \Big)\, d\nu \notag\\
    &= \int c(x,\cdot)-\ppsi +\log\Big( \frac{d\pi_x}{d\nu} \Big) \,d(\eta-\pi_x). 
\end{align}
Recall that \eqref{eq:EOT-WOT-slackness} states that
\begin{align}\label{eq:prfFTEOT2}
\pphi(x) = \int c(x,\cdot) -\ppsi +\log\Big( \frac{d\pi_x}{d\nu} \Big) \,d\pi_x.
\end{align}
Using this, \eqref{eq:prfFTEOT1} inequality simplifies to 
\begin{align}\label{eq:prfFTEOT3}
    0 \le \int c(x,\cdot) -\pphi(x) -\ppsi +\log\Big( \frac{d\pi_x}{d\nu} \Big) \,d\eta.
\end{align}
Applying \eqref{eq:prfFTEOT3} to every $\eta \in \cP(Y)$ that is absolute continuous w.r.t.\ $\nu$ with bounded density, yields the  $\nu$-a.s. inequality 
\begin{align}\label{eq:prfFTEOT4}
    0 \le c(x,\cdot) -\pphi(x) -\ppsi +\log\Big( \frac{d\pi_x}{d\nu} \Big).
\end{align}
Note that \eqref{eq:prfFTEOT2} implies that we have equality when integrating \eqref{eq:prfFTEOT4} w.r.t.\ $\pi_x$. Hence, \eqref{eq:prfFTEOT4} is in fact a $\pi_x$-a.s.\ equality. Rearranging terms and applying the exponential function yields \eqref{eq:prfFTEOT0}. This finishes the proof of \eqref{eq:proddensity}.

Conversely, assume that $\pi \in \cpl(\mu,\nu)$ and that \eqref{eq:proddensity} holds for some functions $(\pphi,\ppsi) \in \cL^1(\mu) \times \cL^1(\nu)$. As the first marginal of $\pi$ is $\mu$, the disintegration of $\pi$ w.r.t.\ $\mu$ satisfies \eqref{eq:prfFTEOT0}. As the second marginal of $\pi$ is $\nu$, we have for $\mu$-almost every $x$
\[
\pphi(x) = -\log\left( \int  \exp(\ppsi-c(x,\cdot))\,d\nu \right).
\]
Then a straightforward calculation show that \eqref{eq:EOT-WOT-slackness} holds true. Hence, we derive optimality of $\pi$ and $(\pphi,\ppsi)$ from Propsition~\ref{prop:EOTWOTdual}.
\end{proof}

\begin{remark}
In entropic optimal transport usually a different dual problem is used, namely
\begin{align}\label{eq:usualEOTdual}
    \sup_{u \in L^1(\mu), v \in L^1(\nu) } \int u \, d\mu + \int v\, d\nu - \int \exp(u+v-c) \,d\mu \otimes \nu +1 .
\end{align}
It is well known (see e.g.\ \cite[Theorem~4.7]{Nu22}) that there is strong duality for this problem. Note that if $(\pphi,\ppsi)$ is admissible in \eqref{eq:EOTWOTdual}, then it is also admissible for \eqref{eq:usualEOTdual}. Moreover, if $(\pphi,\ppsi)$ is optimal in \eqref{eq:EOTWOTdual}, then \eqref{eq:proddensity} yields that $ \int \exp(\pphi+\ppsi-c) \,d\mu\otimes\nu =1$, so $(\pphi,\ppsi)$ also yields the optimal value in \eqref{eq:usualEOTdual}.\hfill$\diamond$

\end{remark}
\subsection{Regularization with a general convex function}\label{sec:RWGCF} 
In this section, we briefly sketch how regularization of optimal transport with a general convex function as studied in \cite{BlSeRo18,DePaRo18,DiGe20} fits into the framework of weak optimal transport. A particularly relevant case is quadratically regularized OT, in which case there are significantly more refined results, see \cite{LoMaMe21, Nu24, GoNu24, GoNu24b, GoNuVa24}. 

Specifically, we consider the problem
\begin{align}\label{eq:WOT-Regularized}
    \inf_{\pi \in \cpl(\mu,\nu)} \int c \,d\pi + \int h\left( \frac{ d\pi}{d\mu \otimes \nu} \right) d \mu \otimes \nu,
\end{align}
where $c: X \times Y \to \R$ is lower semi continuous, lower bounded and bounded from above by integrable functions, i.e.\ $c(x,y) \le a(x)+b(y)$ for $a \in \cL^1(\mu), b \in \cL^1(\nu)$, and $h : \R \to \R$ is a convex function. The case of quadratically regularized OT corresponds to $h(t) = \frac{1}{2} t^2$ if $t\ge 0$ and $h(t)=+\infty$ if $t<0$.

Note that \eqref{eq:WOT-Regularized} is a weak optimal transport problem with cost 
\begin{align}
    C(x,\rho)  = \int c(x,\cdot) \, d\rho + \int h\left( \frac{d\rho}{ d\nu}\right) \,d\nu.
\end{align}
By using the same arguments as in the proof of Proposition~\ref{prop:EOTWOTdual}, we find that the weak optimal transport dual problem for \eqref{eq:WOT-Regularized} is given by
\begin{align}\label{eq:EOTWOTdual_h}
    \sup \left\{\mu (\pphi) +\nu(\ppsi)\, ; \,(\pphi,\ppsi) \in \cL^1(\mu) \times \cL^1(\nu) : \pphi(x) + \rho(\ppsi) 
    \le \int c(x, \cdot)\frac{d\rho}{d\nu} + h\Big(\frac{d\rho}{d\nu} \Big) \, d\nu \right\}.
\end{align}
Moreover, the fundamental theorem of weak optimal transport states that there is strong duality, primal attainment and dual attainment. Further, complementary slackness for this problem reads as: Given $\pi \in \cpl(\mu,\nu)$ and admissible pair $(\pphi,\ppsi)$ the following are equivalent:
\begin{enumerate}
    \item $\pi$ is a primal optimizer and $(\pphi,\ppsi)$ is a dual optimizer
    \item We have for $\mu$-a.e.\ $x$
    \begin{align*}%\label{eq:EOT-WOT-slackness2}
    \int c(x, \cdot) \, d\pi_x + \int h\Big( \frac{d\pi_x}{d\nu} \Big) \,d\nu = \pphi(x) + \int \ppsi \, d\pi_x. 
\end{align*}
\end{enumerate}
Using the methods of the proof of \Cref{prop:EOT_opt_density}, we find that if $\pi$ and $(\pphi,\ppsi)$ are optimal, then
\begin{align}\label{eq:pi_h_bla}
    \frac{d\pi_x}{d\nu}(y) = (h^\ast)'(\alpha(x) + \ppsi(y) -c(x,y)),
\end{align}
where $\alpha(x)$ is chosen such that 
\begin{align}\label{eq:alpha}
    \int (h^\ast)'(\alpha(x) + \ppsi(y) -c(x,y)) \, \nu(dy)=1.
\end{align}

Further note that if $h'(0)=-\infty$ (understood as right-derivative), then $h^\ast$ is strictly increasing, hence $(h^\ast)' >0$, showing that $\spt(\pi) = \spt(\mu \otimes \nu)$. 

We close this section with a remark on the connection between the weak optimal transport dual \eqref{eq:EOTWOTdual_h} an the dual problem used in \cite{DiGe20}, which reads as
\begin{equation}
    \sup_{(u,v) \in \cL^1(\mu) \times \cL^1(\nu)} \int u \, d\mu + \int v\, d\nu - \int h^\ast(u+v-c) \,  d\mu \otimes \nu. 
\end{equation}
If $(f,g)$ is optimal for \eqref{eq:EOTWOTdual_h}, then $(u,v):= (\alpha,g)$, where $\alpha$ is defined according to \eqref{eq:alpha}, is optimal for this problem. This can be seen by comparing \eqref{eq:pi_h_bla} with the complementary slackness condition provided in \cite[Theorem~3.6]{DiGe20}.

\section{Relaxed weak martingale transport}\label{sec:RWMT}
A weak martingale transport problem is a weak transport problem of the form 
\begin{align*}
    {\rm WMT}_C(\mu,\nu) := \inf_{\pi \in \cplm(\mu,\nu)} \int C(x, \pi_x) \, \mu(dx), 
\end{align*}
where $C : \R^d \times \cP_1(\R^d) \to \R$ is a cost function and $\cplm(\mu,\nu)$ is the set of martingale transports from $\mu$ to $\nu$. Note that ${\rm WMT}_C$ is a weak transport problem with cost of the form 
\begin{align}\label{eq:WMOTcost}
    C (x,\rho) =\begin{cases}
        \widehat C(\rho) & \mean{\rho} = x,\\
        \infty & \text{else}.
    \end{cases}
\end{align}

The fundamental theorem of weak optimal transport guarantees primal existence and strong duality for this problem. However, the fact that $C(x,\rho) = + \infty$ whenever $\mean{\rho}  \neq x$ means that the boundedness condition \eqref{eq:WOTbound} is not satisfied. Therefore, \Cref{thm:FTWOT_mainbody} does not guarantee dual attainment; and, as already pointed out in the introduction, dual attainment for (weak) martingale transport in dimension $d >1$ can even fail in very regular settings. 

As a remedy for this, we relax the first marginal condition in the weak transport problem. Specifically, we consider the problem
\begin{align}\label{eq:WMT-Yosida}
\inf_{\eta \in \cP_p(\R^d)}  {\rm T}_\vartheta(\mu,\eta)  +  {\rm WMT}_C(\eta,\nu),    
\end{align}
where $\vartheta : \R^d \to \R$ is a convex function. This problem can be seen as a distributionally robust WMOT problem. In the case of $p=2$ and $\vartheta(x)=|x|^2$ it can also be seen as Wasserstein--Yosida regularization of ${\rm WMT}_C(\cdot,\nu)$.  

The main observation of this section is that problem \eqref{eq:WMT-Yosida} corresponds to the weak transport problem with cost 
\begin{align}\label{eq:WMOTregCost}
    C_\vartheta(x, \rho) = \vartheta(x-\mean{\rho}) + \widehat C(\rho).
\end{align}
For $\vartheta = \chi_{\{0\}}$, we formally recover the cost $C$ introduced in \eqref{eq:WMOTcost} and the relaxed problem \eqref{eq:WMT-Yosida} reduces to WMOT.  

We observe that convexity of $\widehat C$ on fibers of the map $\rho \mapsto \mean{\rho}$, i.e.\ $\widehat C ( (1-t)\rho_1 + t\rho_2) \le (1-t) \widehat C (\rho_1) + t \widehat C (\rho_2)$ whenever $\rho_1,\rho_2 \in \cP_p(\R^d)$ satisfy $\mean{\rho_1}=\mean{\rho_2}$, implies convexity of the cost $C$, but is a too weak condition to guarantee convexity of $C_\vartheta$.\footnote{Convexity of $\widehat C$ is sufficient to guarantee convexity of $C_\vartheta$. Note however that the function $\widehat{C}$ arising in entropic martingale transport (see \Cref{sec:REMOT}) is merely convex on the fibers of $\rho \mapsto \mean{\rho}$.} Therefore, we need to invoke the theory of weak transport with non-convex cost as outlined in \Cref{sec:Relaxed_Primal_Problem}.

The aim of this section is to derive a 'fundamental theorem of relaxed WMOT' from the fundamental theorem of WOT with cost $C_\vartheta$. For this, we need appropriate conditions on $\widehat C$ to guarantee that $C_\vartheta$ satisfies \eqref{eq:WOTbound} and \eqref{eq:WOTcont}. For the boundedness condition \eqref{eq:WOTbound}, this is the existence of a function $b \in \cL^1(\nu)$ and an increasing function $h : [0,\infty) \to [0,\infty)$ such that 
\begin{align}\label{eq:MWOTbound} \tag{{\sf{{BM}}}}
    \widehat C(\rho) \le \rho (b) + \int h\Big(\frac{d\rho}{d\nu} \Big)\, d\nu.
\end{align}
The analogue of the continuity condition is the following: For every increasing sequences $(Y_k)_k$ of Borel sets with $\bigcup_k Y_k=\R^d$ and every $\rho \in \cP_p(\R^d)$ we have
\begin{align}\label{eq:MWOTcont} \tag{{\sf{{CM}}}}
    \widehat C(\rho) \ge \limsup_k \widehat C( \tfrac{1}{\rho(Y_k)} \rho|_{Y_k}).
\end{align}

\begin{theorem}\label{thm:RWMT}
Let $\mu,\nu \in \cP_p(\R^d)$ and let $\widehat C: \cP_p(\R^d) \to [0,\infty]$ be lsc and convex on the fibres of $\rho \mapsto \mean{\rho}$. Suppose that $\widehat C$ satisfies\eqref{eq:MWOTbound} and \eqref{eq:MWOTcont}. Further, let $\vartheta : \R^d \to [0,\infty)$ be a convex function satisfying $\vartheta(x-y) \le a(x) +b(y)$ for some $a \in \cL^1(\mu)$ and $b \in \cL^1(\nu)$. Then we have the following assertions:
\begin{enumerate}[\upshape(i)]
    \item\label{it:RWMT1} The problem \eqref{eq:WMT-Yosida} is equivalent to the relaxed WOT problem with cost $C_\vartheta$, i.e.
    \begin{align}\label{eq:WMOTREG}
    {\overline{{\rm WT}}}_{C_\vartheta}(\mu,\nu) = \min_{\eta \in \cP_p(\R^d)}  {\rm T}_\vartheta(\mu,\eta)  +  {\rm WMT}_C(\eta,\nu)
\end{align}  
  
\item\label{it:RWMT2}We have strong duality and dual attainment, that is
\begin{align}\label{eq:RegWMOT-Dual}
    \min_{\eta \in \cP_p(\R^d)}  {\rm T}_\vartheta(\mu,\eta)  +  {\rm WMT}_C(\eta,\nu) = \max_{g \in L^1(\nu)} \mu( \vartheta \Box g^C) + \nu(g).  
\end{align}

\item\label{it:RWMT3} Both $\eta \in \cP_p(\R^d)$ and $g \in L^1(\nu)$ are optimal in \eqref{eq:RegWMOT-Dual} if and only if
\begin{align}
    {\rm T}_\vartheta(\mu,\eta) &=\mu(\vartheta \Box g^C) - \eta (g^C) \label{eq:RWMT3a} \\
    {\rm WMT_C}(\eta, \nu) &=  \eta(g^C) + \nu(g) \label{eq:RWMT3b}
\end{align}

\item\label{it:RWMT4} Both $P \in \Lambda(\mu,\nu)$ and $g \in L^1(\nu)$ are optimal in \eqref{eq:WMOTREG} if and only if for $P$-a.e.\ $(x,\rho)$
\begin{gather}
    \mean{\rho} \in \partial^{\vartheta} (\vartheta \Box g^C)(x), \label{eq:RWMT41}\\
    \rho \in \arg \min \{ \widehat C(q) - q(g) : q \in \P_1(\R^d), \, q(|g|) < \infty \} \label{eq:RWMT42}.
\end{gather}
\end{enumerate}
\end{theorem}
The interpretation of \eqref{eq:WMOTREG} is that the following viewpoints are equivalent: Relaxing the martingale constraints while the marginals remain fixed and a distributionally robust version of martingale optimal transport (in the first marginal).

Note that the conditions \eqref{eq:RWMT3a} and \eqref{eq:RWMT3b} precisely say that the pair $(\vartheta \Box g^C, g^C)$ is a dual optimizer for the transport problem from $\mu$ to $\eta$ and that $(g^C,g)$ is a dual optimizer for the weak martingale optimal transport problem between $\eta$ and $\nu$. 

We further note the connection between the transforms associated to the weak martingale transport problem and its regularized version. Specifically, we have 
\begin{align}\label{eq:C-theta-trafo}
    g^{C_\vartheta}(x) &= \inf_{m \in \R^d} \inf_{\mean{\rho}=m}  \vartheta(x-m) + C(m,\rho) - \rho(g) \notag \\
    &= \inf_{m \in \R^d} \vartheta(x-m) + g^{ C}(m) = \vartheta \Box g^{ C}(x).
\end{align}
Before proving \Cref{thm:RWMT}, we discuss two relevant instances where problem \eqref{eq:WMT-Yosida} even corresponds to non-relaxed WOT problem and allows us to strengthen the results from \Cref{thm:RWMT}~\eqref{it:RWMT1}.

\begin{proposition}\label{prop:RWMT}
In addition to the assumptions of \Cref{thm:RWMT} suppose one of the following:
\begin{enumerate}[\upshape(a)]
    \item\label{it:RWMT_A} The cost is of the form $C(x,\rho)=C(\rho)$ for a convex lsc function $C : \cP_p(\R^d) \to \R$ 
    \item\label{it:RWMT_B} $\mu$ is absolutely continuous, and $\vartheta$ is of the form $\vartheta(x) = \bar\vartheta(|x|)$ for some increasing, strictly convex function $\bar\vartheta : [0,\infty) \to [0,\infty)$
\end{enumerate}
Then \eqref{eq:WMT-Yosida} is also equivalent to the non-relaxed WOT problem with cost $C_\vartheta$, i.e. 
\begin{align}\label{eq:WMOTREG2}
    {\rm WT}_{C_\vartheta}(\mu,\nu) = \min_{\eta \in \cP_p(\R^d)}  {\rm T}_\vartheta(\mu,\eta)  +  {\rm WMT}_C(\eta,\nu).
\end{align}  
Moreover, there exists an optimizer $\pi \in \cpl(\mu,\nu)$ for ${\rm WT}_{C_\vartheta}(\mu,\nu)$.

If $C$ is strictly convex in $\rho$, then $\pi$ is unique and we have the following optimality condition: $\pi \in \cpl(\mu,\nu)$ is optimal in ${\rm WT}_{C_\vartheta}(\mu,\nu)$ if and only if 
\begin{align}
    % \pi \text{ is optimal in }  {\rm WT}_{C_\vartheta}(\mu,\nu) \iff 
    \begin{cases}
        \eta = (x \mapsto \mean{\pi_x})_\# \mu \text{ is optimal for } \eqref{eq:WMOTREG2},  \\
        \xi = ( x \mapsto(x, \mean{\pi_x})_\# \mu \text{ is optimal for } {\rm T}_\vartheta(\mu,\eta), \\
        \kappa \text{ is optimal for } {\rm WMT}_C(\eta,\nu), \\
        \pi_x = \kappa_{\mean{\pi_x}} \, \mu\text{-a.s.}
    \end{cases}
    \label{eq:pi-opt-dec}
\end{align}
\end{proposition}
We point out that both assumptions in \Cref{prop:RWMT} arise naturally. Case \eqref{it:RWMT_A} arises in the relaxation of the martingale Benaumou--Brenier problem (see \Cref{sec:BS_SBM}) while \eqref{it:RWMT_B} is natural in the context of regularized entropic martingale transport (see \Cref{sec:REMOT}). There are similar results to \eqref{eq:pi-opt-dec} if $C$ is not strictly convex and for the optimal $P \in \Lambda(\mu,\nu)$ in ${\overline{{\rm WT}}}_{C_\vartheta}(\mu,\nu)$, see \Cref{cor:tedious_decomposition} below.

\begin{remark}
The assumptions of \Cref{prop:RWMT} \eqref{it:RWMT_B} are chosen to guarantee the existence of an optimal Monge coupling between $\mu$ and $\eta$. We point out this is also true for a more general class of strictly convex functions (see \cite[Theorem~1.2]{GaMc96}) and also for the function $\vartheta(x)=|x|$ (see e.g.\ \cite[Theorem~2.50]{Vi03}). In these cases the assertions of \Cref{prop:RWMT} remain vaild. Note however that strict convexity of both $C$ and $\vartheta$ is needed to guarantee uniqueness of the optimal $\pi$. \hfill$\diamond$
\end{remark}

\subsection{Application to martingale Benamou--Brenier}\label{sec:BS_SBM}
The weak transport problem associated to martingale Benamou--Brenier is given by
\[
{\rm MBB }(\mu,\nu)= \sup_{\pi \in \cplm(\mu,\nu)}  \int \MCov(\pi_x,\gamma)\,  \mu(dx), 
\]
where 
\[
    \MCov(\rho_1,\rho_2) = \sup_{q \in \cpl(\rho_1,\rho_2)} \int x \cdot y \,q(dx,dy).
\]
In the case that one of the measures $\rho_1,\rho_2$ is centered, $\MCov(\rho_1,\rho_2)$ is the maximal covariance that a pair of random variables $(X_1,X_2)$ with $\law(X_i)=\rho_i$ can admit. In this section, we study an interpolation of this problem and a barycentric transport problem, i.e. 
\begin{equation}\label{eq:BS_SBM_WOT}
   {\rm MBB}_{\rm reg}(\mu , \nu) = \inf_{\pi \in \cpl(\mu,\nu)}
    \int \beta \vartheta(x - \mean{\pi_x}) - \alpha \MCov(\pi_x,\gamma) \, \mu(dx),
\end{equation}
where $\alpha,\beta\geq 0$ are real parameters, $\vartheta:\R^d \to \R$ is a convex function, and $\gamma \in \cP_2(\R^d)$ is a centered absolutely continuous measure, the most relevant example being the standard Gaussian. 

Specifically, if $\vartheta: \R^d \to \R$ is a convex function with unique minimum in 0 (e.g. $\vartheta(x)=|x|$ or $\vartheta(x)=|x|^2$), $\alpha=1$ and $\beta \to \infty$, then \eqref{eq:BS_SBM_WOT} can be seen as a martingale Benamou--Brenier problem with relaxed martingale constraint. Conversely, if $\beta =1$ and $\alpha$ is small, then \eqref{eq:BS_SBM_WOT} can be seen as strictly convex perturbation of the barycentric transport problem (note that the latter is not strictly convex, even if $\vartheta$ itself is strictly convex). 

For notational simplicity we set $\alpha=\beta=1$ from now on. Note that this is wlog by replacing $\gamma$ with $\gamma^\alpha := (x \mapsto \alpha x)_\# \gamma$ and $\vartheta$ with $\vartheta^\beta(x):=\beta \vartheta(x)$. We write $\check\gamma := (x \mapsto -x)_\#\gamma$ for $\gamma$ reflected at the origin. Applying \Cref{thm:RWMT} to ${\rm MBB_{reg}}$ yields

\begin{theorem}\label{thm:SBMbary}
Let $\mu, \nu \in \mathcal{P}_2(\mathbb{R}^d)$, let $\gamma \in \P_2(\R^d)$ be centered and absolutely continuous, 
and let $  \vartheta: \mathbb{R}^d \to \mathbb{R} $ be a convex function. Assume there exist functions \(a \in L^1(\mu)\) and \(b \in L^1(\nu)\), both convex, such that $\vartheta \leq a \Box b$. 
    \begin{enumerate}[\upshape(i)]
        \item\label{it:SBMbary3} We have 
        \begin{align}\label{eq:MBBBT_eta}
            {\rm MBB}_{\rm reg}(\mu , \nu)  = \min_{\eta \in \cP_2(\R^d)} {\rm T}_\vartheta(\mu, \eta) - {\rm MBB}(\eta, \nu). 
        \end{align}
        There exist a unique primal optimizer $\pi$ for ${\rm MBB}_{\rm reg}(\mu , \nu)$ and a unique optimal $\eta$ on the right hand side. Moreover, $\pi$ is optimal if and only if
        \begin{align}
    \begin{cases}
        \eta = (x \mapsto \mean{\pi_x})_\# \mu \text{ is optimal for } \eqref{eq:MBBBT_eta},  \\
        \xi = ( x \mapsto(x, \mean{\pi_x})_\# \mu \text{ is optimal for } {\rm T}_\vartheta(\mu,\eta) \\
        \kappa \text{ is optimal for } {\rm MBB}(\eta,\nu), \\
        \pi_x = \kappa_{\mean{\pi_x}} \, \mu\text{-a.s.}
    \end{cases}
    \label{eq:pi-opt-dec_thm}
\end{align}
     \item\label{it:SBMbary1} We have strong duality and dual attainment: 
        \begin{align}\label{eq:MBBBT_psi}
            {\rm MBB}_{\rm reg}(\mu , \nu)  = \max_{\psi \textrm{ convex, lsc}}
            \mu( \vartheta \Box (\psi^\ast \ast \check\gamma)^\ast) -\nu(\psi).     
        \end{align}
        as well as primal attainment and dual attainment. Moreover, the primal optimizer is unique. 
        \item Both $\eta \in \cP_2(\R^d)$ and $\psi$ are optimal in \eqref{eq:MBBBT_eta} and \eqref{eq:MBBBT_psi}, resp.\ if and only if 
        \begin{align}
             {\rm T}_\vartheta(\mu,\eta)& = \mu(\vartheta\Box (\psi^* * \gamma)^*) - \eta((\psi^* * \gamma)^*)  \label{eq:SBM_char1}  \\
             -{\rm MBB }(\eta,\nu) & = \eta((\psi^* * \gamma)^*) - \nu(\psi)  \label{eq:SBM_char2}
        \end{align}
    \end{enumerate}
\end{theorem}
Note that if $\gamma=N(0,\id)$, then \eqref{eq:SBM_char2} is equivalent to $\kappa=\law(M_0, M_1)$ for a Bass martingale $M_t=\EE[v(B_1)|B_t]$ where $v=\psi^*$, cf.\ \eqref{eq:Bass_definition} and \cite[Theorem~1.4]{BaBeScTs23}). 

We also point out that for $\vartheta(x)=|x|^2$, \eqref{eq:MBBBT_eta} states that ${\rm MBB}_{\rm reg}(\cdot,\nu)$ is the Yosida regularization of ${\rm MBB}(\cdot,\nu)$ in the metric space $(\cP_2(\R^d),\cW_2)$.

\begin{remark}
The fact that it suffices to take the supremum over convex functions in the dual problem, see \eqref{eq:MBBBT_psi}, follows a more general principle observed by Pramenkovi{\'c} \cite{Pr24}: If the function $\widehat C$ is decreasing w.r.t.\ convex order (i.e.\ if $\rho_1 \le_c \rho_2$, then $\widehat C(\rho_1) \ge \widehat C(\rho_2)$), the dual problem can be restricted to convex functions.\hfill$\diamond$
\end{remark}

\medskip
In order to derive \Cref{thm:SBMbary} from \Cref{thm:RWMT} we need to calculate the $C$-transform associated to the cost associated to ${\rm MBB}$.

\begin{proposition}\label{prop:MB3T-Trafo}
Let $\psi : \R^d \to (-\infty,\infty]$ be convex. Then $(-\psi)^C$ is convex as well and
\[
    ((-\psi)^C)^{\ast\ast} = (\psi^\ast \ast \check\gamma)^\ast.
\]
In particular, $(\psi^\ast \ast \check\gamma)^\ast$ is proper if and only if $(-\psi)^C$ is proper. 
\end{proposition}
\begin{proof}
Using the definition of the $C$-transform yields
\begin{align*}%\label{eq:hat}
    (-\psi)^C(m) = \inf_{\substack{\rho \in \mathcal P_2(\R^d) \\ \mean{\rho} = m}} \inf_{\pi \in \cpl(\rho,\gamma)} \int \psi(y) -y \cdot  z  \, \pi(dy,dz).
\end{align*}
Next, we calculate the convex conjugate of the $C$-transform
\begin{align*}
    ((-\psi)^C)^{\ast}(x) &= \sup_{m \in \R^d} x \cdot m + \sup_{\substack{\rho \in \mathcal P_2(\R^d) \\ \mean{\rho} = m}} \sup_{\pi \in \cpl(\rho,\gamma)} \int y \cdot  z - \psi(y) \, \pi(dy,dz) \\
    &= \sup_{ \pi \in \cpl(\cdot, \gamma) } \int y \cdot ( z +x) - \psi(y) \, \pi(dy,dz) \\
    &= \int \sup_{y \in \R^d} y \cdot ( z +x) - \psi(y) \,  \gamma(dz) 
    = \int \psi^\ast( z +x) \, \gamma(dz)
    = (\psi^\ast \ast \check\gamma)(x).
\end{align*}
As $\psi$ is convex, all integrals in this calculation are well defined. Taking the convex conjugate on both sides of this equality yields the claim. 
\end{proof}

\begin{proof}[Proof of \Cref{thm:SBMbary}]
We need to observe that the cost function of ${\rm MBB}_{\rm reg}$ is convex in the second argument. Then \Cref{thm:SBMbary} follows directly from \Cref{thm:RWMT} and \Cref{prop:RWMT}, Case \eqref{it:RWMT_A}. Note that \Cref{prop:MB3T-Trafo} gives the explicit expression of the $C$-transform arising \Cref{thm:RWMT}. 

For the strict convexity of the cost, it suffices to prove that $\rho \mapsto \MCov(\gamma,\rho)$ is strictly convex. This follows from Brenier's theorem which asserts uniqueness of the primal optimizer: Suppose that $\rho^1,\rho^2$ are such that $\frac{1}{2}\MCov(\rho^1,\gamma) +  \frac{1}{2}\MCov(\rho^2,\gamma) = \MCov(\frac{1}{2}(\rho^1+\rho^2),\gamma)$. If $\xi^i \in \cpl(\gamma,\rho^i)$, $i \in \{1,2\}$ are optimal, then $\xi = \frac{1}{2}(\xi^1+\xi^2) \in \cpl(\gamma,\frac{1}{2}(\rho^1+\rho^2))$ is optimal. As $\gamma$ is absolutely continuous, Brenier's theorem yields that $\xi$ is supported on the graph of a function. This implies $\xi^1=\xi^2$ and hence $\rho^1=\rho^2$. 
\end{proof}

\subsection{Application to entropic martingale transport}\label{sec:REMOT}

The final section is dedicated to the entropic martingale transport problem and its relaxation.
More specifically, given $\eta,\nu \in \P_1(\R^d)$, a measurable cost $c : \R^d \times \R^d \to [0,\infty)$, and the regularization parameter $\eps > 0$, the entropic martingale transport problem reads as
\begin{equation}
    \label{eq:EMOT}
    {\rm EMT}_{c,\eps} (\eta,\nu) = \inf_{ \kappa \in \cplm(\eta,\nu) } \int c(m,y) \, \kappa(dm,dy) + \eps H(\pi | \eta \otimes \nu),
\end{equation}
which corresponds to the weak transport problem with cost $C : \R^d \times \P_1(\R^d) \to \R \cup \{ +\infty \}$ defined as in \eqref{eq:WMOTcost} where
\begin{equation}
    \label{eq:hatC.EMOT}
    \widehat C(\rho) = \int c(\mean{\rho},y) \, \rho(dy) + \eps H(\rho | \nu).
\end{equation}
In this section, we study the relaxation as proposed in \eqref{eq:WMT-Yosida} where we penalize deviation from the mean according to a convex function $\vartheta : \R^d \to \R$ and study the relaxation
\begin{equation}
    \label{eq:REMOT}
    \inf_{\eta \in \P_1(\R^d)}
    {\rm T}_\vartheta(\mu,\eta) + {\rm EMT}_{c,\eps}(\eta,\nu).
\end{equation}
Throughout this section, we denote by $C_\vartheta$ the associated weak transport cost given by
\begin{equation}
    \label{eq:Ctheta.EMOT}
    C_\vartheta(x,\rho) = \vartheta(x - \mean{\rho}) + \int c(\mean{\rho},y) \, \rho(dy) + \eps H(\rho | \nu).
\end{equation}
Relaxed entropic martingale transport is an instance of a relaxed weak martingale transport problem.
Under the natural assumptions of the next proposition, we have that the structural results of \Cref{thm:RWMT} also hold in the current setting.

\begin{proposition}\label{prop:REMOT}
    Let $\mu,\nu \in \P_p(\R^d)$, let $\vartheta : \R^d \to [0,\infty)$ be convex, and let $c : \R^d \times \R^d \to [0,\infty)$ be lsc.
    Suppose that there are functions $a \in \cL^1(\mu)$ and convex $b \in \cL^1(\nu)$ such that 
    \[
        \vartheta(x-y) \le a(x) + b(y) \quad \text{and} \quad \int c(\mean{\rho},y) \, \rho(dy) \le \rho(b).
    \]
    Then, $\widehat C$ given in \eqref{eq:hatC.EMOT} satisfies the assumptions of \Cref{thm:RWMT}.
    In particular, the conclusions of \Cref{thm:RWMT} hold in the current setting for the relaxed entropic martingale transport problem.
\end{proposition}

\begin{proof}
    It remains to check that $\widehat C$ satisfies the assumptions of \Cref{thm:RWMT}.
    Since $c$ is non-negative and lsc, the same holds true for the map $\pi \mapsto \pi(c)$ for $\pi \in \P(\R^d \times \R^d)$.
    Furthermore, the map $\rho \mapsto \delta_{\mean{\rho}} \otimes \rho$ is continuous from $\P_1(\R^d)$ to $\P(\R^d \times \R^d)$.
    Hence, $\rho \mapsto \int c(\mean{\rho},y) \, \rho(dy)$ is lsc on $\P_1(\R^d)$ as concatenation of these maps.
    As the relative entropy is lsc on $\P(\R^d) \times \P(\R^d)$, we conclude that
    \[
        \widehat C(\rho) = \int c(\mean{\rho},y) \, \rho(dy) + \eps H(\rho | \nu)
    \]
    is lsc on the domain $\P_1(\R^d)$.
    Letting $h(x) = \eps x \log(x)$, we have that $\widehat C$ satisfies \eqref{eq:MWOTbound}.
    Finally, the continuity property \eqref{eq:MWOTcont} follows by the same argument as in the proof of \Cref{prop:EOTWOTdual}.
\end{proof}

When we assume mild regularity properties of the cost $c$ and $\vartheta$, we are able to derive a more precise structure theorem that incorporates structural aspects of the entropic (martingale) optimal transport, namely that optimal martingale couplings are of Gibbs type.

\begin{theorem} \label{thm:REMOT}
    In addition to the assumptions of \Cref{prop:REMOT} suppose that $c$ is Lipschitz, $\mu$ is absolutely continuous and $\vartheta$ is of the form $\vartheta(x) = \bar\vartheta(|x|)$ for some increasing, strictly convex function $\bar\vartheta : [0,\infty) \to [0,\infty)$.
    Then ${\rm WT}_{C_\vartheta}(\mu,\nu) = \overline{\rm WT}_{C_\vartheta}(\mu,\nu)$ and both problems have a unique minimizer.
    
    Moreover, $\pi \in \cpl(\mu,\nu)$ and $g \in \cL^1(\nu)$ are primal and dual optimizers of ${\rm WT}_{C_\vartheta}(\mu,\nu)$ if and only if there is $\Delta : \R^d \to \R^d$ measurable such that
    \begin{gather} \label{eq:thm.REMOT.subdiff}
        \mean{\pi_x} \in \partial^\vartheta(\vartheta \Box g^C)(x) \quad \mu\text{-a.s.,} \\
        \label{eq:thm.REMOT.gibbs}
        \frac{d\kappa}{d\eta \otimes \nu}(m,y) = \exp\Big( \frac{g^C(m) + g(y) + \Delta(m) \cdot (y - m) - c(m,y)}{\eps} \Big),
    \end{gather}    
    where $\kappa = ((x,y) \mapsto (\mean{\pi_x},y))_\# \pi$ and $\eta = (x \mapsto \mean{\pi_x})_\# \mu$.
\end{theorem}

\begin{proof}
    First, note that $\widehat C$ is strictly convex on the fibers $\{ \rho \in \P_1(\R^d) : \mean{\rho} = m \}$ for $m \in \R^d$.
    Thanks to \Cref{prop:REMOT} and the additional assumptions on $\mu$ and $\vartheta$, we can invoke \Cref{prop:RWMT}.
    It remains to show the characterization of optimality given in \eqref{eq:thm.REMOT.subdiff} and \eqref{eq:thm.REMOT.gibbs}.
    
    To this end, recall \eqref{eq:def:J} and that we have by \Cref{thm:RWMT} \eqref{it:RWMT4} that $\pi \in \cpl(\mu,\nu)$ and $g \in \cL^1(\nu)$ are primal and dual optimizers if and only if for $J(\pi)$-a.e.\ $(x,\rho)$ we have \eqref{eq:RWMT41} and \eqref{eq:RWMT42}.
    Clearly, \eqref{eq:RWMT41} holds $J(\pi)$-almost surely if and only if \eqref{eq:thm.REMOT.subdiff} holds.
    Hence, it suffices to show that $\kappa = ((x,y) \mapsto (\mean{\pi_x},y))_\# \pi \in \cplm(\eta,\nu)$ is optimal for ${\rm EMT}_{c,\eps}(\eta,\nu)$ if and only if there is $\Delta : \R^d \to \R^d$ measurable such that $\kappa$ is given by \eqref{eq:thm.REMOT.gibbs}.

    On the one hand, if $\kappa$ is as in \eqref{eq:thm.REMOT.gibbs}, then we have that for $\eta$-a.e.\ $m$
    \[
        \kappa_m \in \arg \min \{ \rho(c(m,\cdot) - g) + \eps H(\rho | \nu) - \Delta(m) \cdot (\mean{\rho} - m) : \rho \in \P_1(\R^d), \, \rho(|g|) < \infty \},
    \]
    because $\eps H(\rho | \kappa_m)$ and $\rho(c(m,\cdot) - g) + \eps H(\rho | \nu)  - \Delta(m) \cdot (\mean{\rho} - m)$ differ as a function of $m$ just by a constant.
    Since $\mean{\kappa_m} = m$ it follows now directly that for $\eta$-a.e.\ $m$
    \[
        \kappa_m \in \arg \min \{ \rho(c(m,\cdot) - g) + \eps H(\rho | \nu) : \rho \in \P_1(\R^d), \, \mean{\rho} = m, \, \rho(|g|) < \infty \}.
    \]
    Hence, $g^C = \kappa_m(c(m,\cdot) - g) + \eps H(\kappa_m|\nu)$ for $\eta$-a.e.\ $m$ and we conclude by \Cref{prop:slack} that $\kappa$ is the optimizer of ${\rm EMT}_{c,\eps}(\eta,\nu)$.

    On the other hand, if $\pi \in \cpl(\mu,\nu)$ and $g \in \cL^1(\nu)$ are primal and dual optimizers of ${\rm WT}_{C_\vartheta}(\mu,\nu)$, then $\kappa = ((x,y) \mapsto (\mean{\pi_x},y))_\# \pi \in \cplm(\eta,\nu)$ and $g \in \cL^1(\nu)$ are primal and dual optimizers of ${\rm EMT}(\eta,\nu)$.
    Wlog we assume that $\nu$ is not a Dirac measure.
    In order to find $\Delta$, we define the auxiliary function
    \[
        F(x,m) := \inf \{ \rho(c(x,\cdot) - g) + \eps H(\rho | \nu) : \mean{\rho} = m, \, \rho(|g|) < \infty \}.
    \]
    From the definition, we see that $F(x,\cdot)$ is convex, finite on $D := {\rm relint}({\rm co}({\rm supp}(\nu)))$ (the relative interior of the convex hull of the support of $\nu$), and $F(x,x) = g^C(x)$.    
    It follows from the Arsenin--Kunugui selection theorem (see \cite[Theorem~18.8]{Ke95}) that there exists a measurable map $\Delta :  D \to \R^d$ such that $\Delta(x) \in \partial F(x,\cdot)(x)$, where we extend $\Delta$ to $\R^d$ by setting it $0$ outside of $D$.
    We have for every $x \in D$ and $\rho \in \P_1(\R^d)$ with $\mean{\rho} = m$ and $\rho(|g|) < \infty$ that
    \[
        F(x,x) + \Delta(x) \cdot (m - x) 
        \le F(x,m) \le \rho(c(x,\cdot) - g) + \eps H(\rho | \nu),
    \]
    where we have $\eta$-a.s.\ equality for $\rho = \kappa_x$, since $\eta(D) = 1$ by \Cref{lem:relint}.
    We can proceed as in the first part of the proof of \Cref{prop:EOT_opt_density}, which yields that
    \[
        \frac{d\kappa}{d\eta \otimes \nu} = \exp\Big( \frac{g^C(m) + \ppsi(y) + \Delta(m) \cdot (y-m) - c(m,y)}{\eps} \Big),
    \]
    which completes the proof.
\end{proof}

\begin{lemma} \label{lem:relint}
    Under the assumption of \Cref{prop:REMOT}, the optimizer $P \in \Lambda(\mu,\nu)$ for $\overline{\rm WT}_{C_\vartheta}(\mu,\nu)$ satisfies $\rho \ll \nu$ and $\nu \ll \rho$ for $P$-a.e.\ $(x,\rho)$.
    In particular, the measure $((x,\rho) \mapsto \mean{\rho})_\# P$ is concentrated on the relative interior of ${\rm co}({\rm supp}(\nu))$.
\end{lemma}

\begin{proof}
    Since the reasoning in \Cref{cor:Cmonotonicity} also works for $P \in \Lambda(\mu,\nu)$ that are optimal for $\overline{\rm WT}_{C_\vartheta}(\mu,\nu)$, we find that $P$ is concentrated on a $C$-monotone set $\Gamma \subseteq \R^d \times \P_1(\R^d)$.
    We claim that $\rho_0 \ll \rho_1$ and $\rho_1 \ll \rho_0$.
    As the intensity of the second marginal of $P$ is $\nu$, this yields that $\rho$ is equivalent to $\nu$ $P$-a.s.\ and the assertion of the lemma readily follows.
    
    To this end, fix $(x_i,\rho_i) \in \Gamma$ with $m_i = \mean{\rho_i}$, $H(\rho_i|\nu) < \infty$ and $b \in \cL^1(\rho_i)$, $i \in \{0,1\}$.
    For the sake of contradiction, suppose that $\rho_0$ and $\rho_1$ are not equivalent.
    Next, we define a curve of measures by
    \[
        \rho_t := \rho_0 + t (\rho_1 - \rho_0),
    \]
    where $t \in [0,1]$ and set $m_t = \mean{\rho_t}$.
    By $C$-monotonicity of $\Gamma$, we find the inequality
    \begin{equation}
        \label{eq:47.11}
        C_\vartheta(x_0,\rho_0) + C_\vartheta(x_1,\rho_1) \le C_\vartheta(x_0,\rho_t) + C_\vartheta(x_1,\rho_{1-t}).
    \end{equation}
    We subtract the left-hand side from the right-hand side, divide by $t$ and send $t \searrow 0$.
    Observe
    \begin{equation} \label{eq:est.0}
        \limsup_{t \searrow 0} \frac1t \Big( \big| \vartheta(x_0 - m_t) - \vartheta(x_0 - m_0) \big| + \big| \vartheta(x_1 - m_{1-t}) - \vartheta(x_1 - m_1) \big| \Big) < \infty,
    \end{equation}
    since $\vartheta$ is a real-valued convex function.
    Since $c(\cdot,y)$ is non-negative and Lipschitz for some constant $L \ge 0$, we find that $c(m_t,\cdot) \le c(m_0,\cdot) \wedge c(m_1,\cdot) + L |m_1 - m_0|$ and
    \begin{multline}
        \label{eq:est.1}
        \limsup_{t \searrow 0} \frac1t \Big( \int c(m_t,\cdot) \, d\rho_t - \int c(m_0,\cdot) \, d\rho_0 \Big) \\ 
        \le
        \limsup_{t \searrow 0} \frac1t \Big( \int c(m_t, \cdot) - c(m_0, \cdot) \, d\rho_0 \Big) + \limsup_{t \searrow 0} \int c(m_t,\cdot) d(\rho_1 - \rho_0)\\
        \le L + \rho_1(b) + L |m_1 - m_0| < \infty.
    \end{multline}
    Similarly, we find
    \[
        \limsup_{t \searrow 0} \frac1t \Big( \int c(m_{1-t},\cdot) \, d\rho_{1-t} - \int c(m_1,\cdot) \, d\rho_1 \Big) < \infty. 
    \]
    Finally, we deal with the entropy terms in $C_{\vartheta}$.
    For $(x,y) \in \R$ we have the inequality $\log(y)y - \log(x)x \ge (\log(x) + 1)(y - x)$, which shows that $\log\Big( \frac{d\rho_t }{d\nu} \Big) \vee 0 \in \cL^1(\rho_0 + \rho_1)$.
    On the other hand, $\log(\frac{d\rho_t}{d\nu}) \ge \log(\frac{d\rho_0}{d\nu}) + \log(1 - t)$. Combining these two inequalities yields
    \[
        H(\rho_0| \nu) + \log(1 - t) \le \int \log\Big( \frac{d\rho_t}{d\nu} \Big) d\rho_0 \le H(\rho_0|\nu),
    \]
    and $\log(\frac{d\rho_t}{d\nu}) \in \cL^1(\rho_0)$ for $t \in [0,1)$.
    Hence,
    \begin{align*}
        H(\rho_t|\nu) - H(\rho_0|\nu) &=
        \int \log\Big( \frac{d\rho_t}{d\nu} \Big) - \log\Big( \frac{d\rho_0}{d\nu} \Big) \, d\rho_0 +
        t \int \log\Big( \frac{d\rho_t}{d\nu} \Big)\, d(\rho_1 - \rho_0)
        \\ &\le t \int \log\Big( \frac{d\rho_t}{d\nu} \Big) \, d(\rho_1 - \rho_0)
    \end{align*}
    where we used that $(\log(y) - \log(x))x \le y - x$. Dividing both sides by $t$, we find 
    \[
        \limsup_{t \searrow 0} \Big( H(\rho_t|\nu) - H(\rho_0|\nu) \Big) \le \int \log\Big( \frac{d\rho_0}{d\nu} \Big) \, d(\rho_1 - \rho_0),
    \]
    by dominated convergence.
    We note that the right-hand side is $-\infty$ if $\rho_0$ is not absolutely continuous w.r.t.\ $\rho_1$.
    Similarly, we find that
    \[
        \limsup_{t \searrow 0} H(\rho_{1-t}| \nu) - H(\rho_0|\nu) \le \int \log\Big( \frac{d\rho_1}{d\nu} \Big) \, d(\rho_0 - \rho_1),
    \]
    where the right-hand side is $-\infty$ if $\rho_1$ is not absolutely continuous w.r.t.\ $\rho_0$. 
    Summarizing, using \eqref{eq:47.11} we find the contradiction
    \[
        0 \le \limsup_{t \searrow 0} \frac1t\Big( C_{\vartheta}(x_0,\rho_t) + C_\vartheta(x_1,\rho_{1-t}) - C_\vartheta(x_0,\rho_0) - C_\vartheta(x_1,\rho_1) \Big) \le -\infty,
    \]
    if $\rho_0 \not\ll \rho_1$ or $\rho_1 \not\ll \rho_0$.
    This concludes the proof.
\end{proof}

\subsection{Proof of \Cref{thm:RWMT} and  \Cref{prop:RWMT}}\label{sec:RWMT_proofs}

\begin{proof}[Proof of \Cref{thm:RWMT} \eqref{it:RWMT1}]
We first show that the left-hand side of \eqref{eq:WMOTREG} is bigger than the right-hand side. To this end, let $P \in \Lambda(\mu,\nu)$. We then write $\eta = ((x,\rho) \mapsto \mean{\rho})_\# P$ and define the couplings $\xi = ((x,\rho) \mapsto (x,\mean{\rho}))_\# P$ and $\kappa(dm,dy) = \int \delta_{\mean{\rho}}(dm)\rho(dy) \,P(d\rho)$.
    Observe that $\xi \in \cpl(\mu,\eta)$ and $\kappa \in \cplm(\eta,\nu)$.
    Using convexity of $C$, we find
    \begin{align*}
        \int C_\vartheta  \, dP &= \int \vartheta(x - m) \, d\xi + \int C(\mean{\rho}, \rho)  \, P(dx, d\rho)  \\
        &\ge \int \vartheta(x - m) \, d\xi + \int C(m,\kappa_m) \, \eta(dm) \\
        &\ge {\rm T}_\vartheta(\mu,\eta) + {\rm WMT}_C(\eta, \nu). 
    \end{align*}

    For the reverse inequality, let $\eta \in \cP_p(\R^d)$ and suppose that $\eta \le_c \nu$ (otherwise there is nothing to prove as ${\rm WMT}_C(\eta,\nu)=+\infty$ in this case). Let $\xi \in \cpl(\mu,\eta)$ and $\kappa \in \cplm(\eta,\nu)$ be optimal for ${\rm T}_\vartheta(\mu,\eta)$ and ${\rm WMT}_C(\eta,\nu)$ respectively. We set $P = ((x,m) \mapsto (x,\kappa_m))_\# \xi$ and observe that $P \in \Lambda(\mu,\nu)$. Using $P$, we estimate
    \begin{align*}
        {\rm WT}_{C_\vartheta}(\mu,\nu) &\le \int C_\vartheta \, dP = \int h(x-m) + C(m, \kappa_m)\, \xi(dx,dm) \\
        &= {\rm T}_\vartheta(\mu, \eta) + {\rm WMT}_C(\eta, \nu) \qedhere 
    \end{align*}
\end{proof}
By going through the constructions given in the previous proof and in particular checking the equality cases of the estimates, it is straightforward to derive 
\begin{corollary}\label{cor:tedious_decomposition}
Let $P \in \Lambda(\mu,\nu)$ and consider the following conditions:
\begin{enumerate}[\upshape(a)]
    \item $\eta = ((x,\rho) \mapsto \mean{\rho})_\# P$ is optimal for \eqref{eq:WMOTREG}.
    \item $\xi = ( (x,\rho)\mapsto(x, \mean{\rho}))_\# P$ is optimal for ${\rm T}_\vartheta(\mu,\eta)$ 
    \item $\kappa$ is optimal for ${\rm WMT}_C(\eta,\nu)$,
    \item $\rho = \kappa_{\mean{\rho}} \, P\text{-a.s.}$
\end{enumerate}
We then have: 
\begin{enumerate}[\upshape(i)]
    \item If $P$ satisfies (a)--(d) it is optimal in ${\overline{{\rm WT}}}_{C_\vartheta}(\mu,\nu)$.
    \item If $P$ is optimal in ${\overline{{\rm WT}}}_{C_\vartheta}(\mu,\nu)$, it satisfies (a)--(c)
    \item If $P$ is optimal in ${\overline{{\rm WT}}}_{C_\vartheta}(\mu,\nu)$ and $C(x,\cdot)$ is strictly convex, $P$ satisfies (a)--(d)
    \item There exists an optimizer $P$ for ${\overline{{\rm WT}}}_{C_\vartheta}(\mu,\nu)$ that satisfies (a)--(d)
\end{enumerate}
\end{corollary}

\begin{proof}[Proof of \Cref{thm:RWMT} \eqref{it:RWMT2} and \eqref{it:RWMT3}]

We first observe that $C_\vartheta$ satisfies the conditions of \Cref{thm:FTWOT_mainbody_L} (FTWOT, non-convex). It is easy to check that $C$ satisfying \eqref{eq:MWOTcont} implies that $C_\vartheta$ satisfies \eqref{eq:WOTcont}. Moreover, $C_\vartheta$ satisfies the growth condition as \eqref{eq:MWOTbound} yields
\begin{align*}
    C_\vartheta(x,\rho) \le \int \vartheta(x-y) \, \rho(dy) +  \rho(b) + \int h\Big(  \frac{d \rho}{d\nu} \Big) \, d\nu \le a(x) +\rho(2b) +  \int h\Big(  \frac{d \rho}{d\nu} \Big) \, d\nu. 
\end{align*}%\LR{2. ungleichung aktuell $b$ convex für Jensen notwendig, oder?}
Therefore, the non-convex fundamental theorem of weak optimal transport (\Cref{thm:FTWOT_mainbody_L}) yields duality and dual attainment for the weak transport problem with cost $C_\vartheta$. Noting that the $C_\vartheta$-transform is first applying the $C$-transform and then inf-convolving with $\vartheta$ (cf. \eqref{eq:C-theta-trafo}), we obtain \eqref{eq:RegWMOT-Dual}. 

Next, we prove the optimality criterion \eqref{it:RWMT3}. Suppose that \eqref{eq:RWMT3a} and \eqref{eq:RWMT3b} are satisfied. We then have 
\[
{\overline{{\rm WT}}}_{C_\vartheta}(\mu,\nu) \ge \mu(\vartheta \Box g ^C ) + \nu(g) = {\rm T}_\vartheta(\mu,\eta) + {\rm WMT}_C(\eta,\nu) \ge {\overline{{\rm WT}}}_{C_\vartheta}(\mu,\nu),
\]
where the first inequality follows because $(\vartheta\Box g^C,g)$ is an admissible dual pair, the equality is obtained by adding \eqref{eq:RWMT3a} and \eqref{eq:RWMT3b} and the last inequality follows from \eqref{eq:WMOTREG}. Hence, all of these inequalities are in fact equalities, i.e. $\eta$ is optimal in \eqref{eq:WMOTREG} and $g$ is optimal in the dual problem of ${\overline{{\rm WT}}}_{C_\vartheta}(\mu,\nu)$.

Conversely, assume that $g \in \cL^1(\nu)$ and $\eta \in \cP_p(\R^d)$ are optimal. Then we find 
\begin{align}\label{eq:prf:RWMTiii}
    {\rm T}_\vartheta(\mu,\eta) + {\rm WMT}_C(\eta,\nu) = {\overline{{\rm WT}}}_{C_\vartheta}(\mu,\nu)  = \mu(\vartheta \Box g ^C ) -\eta(g^C) + \eta(g^C)   + \nu(g). 
\end{align}
As $ (\vartheta \Box g ^C,g^C)$ and $(g^C, g) $ are admissible in the respective dual problems, 
\[
    \mu(\vartheta \Box g ^C) \le {\rm T}_\vartheta(\mu,\eta) \quad  \text{and} \quad \eta(g^C)   + \nu(g) \le {\rm WMT}_C(\eta,\nu). 
\]
By \eqref{eq:prf:RWMTiii}, these inequalities are in fact equalities.
\end{proof}

\begin{proof}[Proof of \Cref{prop:RWMT}, Case \eqref{it:RWMT_A}]
The key observation is that the assumption $C(x,\rho)=C(\rho)$ with $C$ convex guarantees that $C_\vartheta$ is convex in the second argument. \eqref{eq:WMOTREG2} follows from \eqref{eq:WMOTREG} because 
 ${\overline{{\rm WT}}}_{C_\vartheta}(\mu,\nu) = {\rm WT}_{C_\vartheta}(\mu,\nu)$, cf.\ \eqref{eq:Relax_Eq}. Moreover, \Cref{thm:FTWOT_mainbody}~\eqref{it:FTWOT_main1} guarantees the existence of an optimizer $\pi \in \cpl(\mu,\nu)$.

Now suppose that $C(x,\cdot)$ is strictly convex. Then $C_\vartheta(x,\cdot)$ is strictly convex as well and therefore the optimal $\pi$ is unique. 

It remains to prove the optimality criterion \eqref{eq:pi-opt-dec}. Using strict convexity of $C(x,\cdot)$, \Cref{cor:tedious_decomposition} and the fact that $J(\pi) := (x \mapsto (x, \pi_x))_\#\mu$ is optimal if $\pi$ is optimal (see \Cref{sec:Relaxed_Primal_Problem}), yield the following: $\pi$ is optimal for ${\rm WT}_{C_\vartheta}(\mu,\nu)$ if and only if $J(\pi)$ is optimal for ${\overline{{\rm WT}}}_{C_\vartheta}(\mu,\nu)$ if and only if $J(\pi)$ satisfies the conditions (a)--(d) in \Cref{cor:tedious_decomposition}. It is straightforward to check that the right-hand side of \eqref{eq:pi-opt-dec} is equivalent to the conditions (a)--(d) if $P=J(\pi)$. 
\end{proof}

\begin{proof}[Proof of \Cref{prop:RWMT}, Case \eqref{it:RWMT_B}]
As ${\overline{{\rm WT}}}_{C_\vartheta}(\mu,\nu) \le  {\rm WT}_{C_\vartheta}(\mu,\nu)$, equation \eqref{eq:WMOTREG} yields that 
\[
     {\rm WT}_{C_\vartheta}(\mu,\nu) \ge \min_{\eta \in \cP_p(\R^d)}  {\rm T}_\vartheta(\mu,\eta)  +  {\rm WMT}_C(\eta,\nu).
\]
Let $\eta \in \cP_p(\R^d)$, $\xi \in \cpl(\mu,\eta)$ and $\kappa \in \cplm(\eta,\nu)$ all be optimal. As the assumptions of the Gangbo--McCann theorem are satisfied, $\xi = (\id,T)_\#\mu$ for some Borel map $T$. We set $\pi(dx,dy) = \kappa_{T(x)}(dy)$. Observe that $\pi \in \cpl(\mu,\nu)$ and $\mean{\pi_x}= \mean{\kappa_{T(x)}}=T(x)$. 

\begin{align*}
    {\rm WT}_{C_\vartheta}(\mu,\nu) &\le \int \vartheta(x-\mean{\pi_x}) + C(\mean{\pi_x}, \pi_x) \, \mu(dx) \\
    & \le \int \vartheta(x-T(x)) \, \mu(dx) + \int C(m, \kappa_m) \, \eta(dm) \\
    &= {\rm T}_\vartheta(\mu,\eta) + {\rm WMT}_C(\eta,\nu).
\end{align*}
This proves \eqref{eq:WMOTREG} and attainment for $ {\rm WT}_{C_\vartheta}(\mu,\nu)$.

In the case that $C(x,\cdot)$ is strictly convex, the proof of uniqueness of $\pi$ and \eqref{eq:pi-opt-dec} follow line by line as in the proof of \Cref{prop:RWMT}, Case \eqref{it:RWMT_A}. 
\end{proof}
\begin{proof}[Proof of \Cref{thm:RWMT} \eqref{it:RWMT4}]
Suppose that $P \in \Lambda(\mu,\nu)$ and $g \in \cL^1(\nu)$ are optimal. By \Cref{cor:tedious_decomposition}, $\eta = ((x,\rho) \mapsto \mean{\rho})_\# P$ is optimal for \eqref{eq:WMOTREG}, 
     $\xi = ( (x,\rho)\mapsto(x, \mean{\rho}))_\# P$ is optimal for ${\rm T}_\vartheta(\mu,\eta)$, and $\kappa$ is optimal for ${\rm WMT}_C(\eta,\nu)$. By \Cref{thm:RWMT} \eqref{it:RWMT3}, $(\vartheta \Box g^C, g^C)$ is a dual optimizer for ${\rm T}_\vartheta(\mu,\eta)$ and $(g^C,g)$ is a dual optimizer for ${\rm WMT}_C(\eta,\nu)$. Hence, \eqref{eq:RWMT41} follows from the complementary slackness criterion for ${\rm T}_\vartheta(\mu,\eta)$ and \eqref{eq:RWMT42} follows from the complementary slackness criterion for ${\rm WMT}_C(\eta,\nu)$.
     
     Conversely, suppose that $P \in \Lambda(\mu,\nu)$ and $g \in \cL^1(\nu)$ satisfy \eqref{eq:RWMT41} and \eqref{eq:RWMT42}. Adding up these conditions, implies that $P(dx,d\rho)$-a.s.
     \[
     (\vartheta \Box g^C)(x) + \rho(g) = \widehat C(\rho) + \vartheta(x-\mean{\rho}).
     \]
     This is precisely the complementary slackness condition for ${\overline{{\rm WT}}}_{C_\vartheta}(\mu,\nu)$. Hence, \Cref{thm:FTWOT_mainbody_L} \eqref{it:FTWOT_main4_L} yields optimality of $P$ and $g$. 
\end{proof}

\appendix

\section{Auxiliary results from convex analysis}\label{sec:convex_ana} 
The aim of this appendix is to give a brief recap of results from convex analysis (for more details see e.g.\ \cite{BaCo11, Ro70, BoVa10}) and to derive a few auxiliary results that were used throughout. 
\subsection{Convex functions: subdifferential, conjugation and infimal convolution}\label{sec:app1}
A function $f: \R^d \to \R \cup \{+\infty \}$ is convex if $f((1-t)x+ty) \le (1-t)f(x)+tf(y)$ for all $x,y \in \R^d$ and $t \in [0,1]$. The domain of a function $f$ is $\dom(f) = \{x \in \R^d : f(x)<\infty \}$. Moreover, $f$ is called proper if $\dom(f) \neq \emptyset$ and finite if $\dom(f)=\R^d$. Every convex function is continuous in the interior of its domain. 

The convex conjugate (or Legendre transform) of a proper function $f: \R^d \to \R \cup \{+\infty\}$ is defined as
\begin{align}\label{eq:def_legendre}
f^\ast(y) = \sup_{x \in \R^d} \{ \langle x,y \rangle - f(x) \}.
\end{align}
Note that $f^{\ast\ast}$ is the largest lsc convex function that is smaller than $f$. The Fenchel--Moreau theorem states that $f = f^{\ast\ast}$ if and only if $f$ is convex and lsc. 

The subdifferential of a convex function $f: \R^d \to \R \cup \{+\infty\}$ at $x \in \R^d$ is 
\[
\partial f (x) = \{y \in \R^d : f(x') \ge f(x) + \langle x'-x,y \rangle \text{ for every } x' \in \R^d \}.
\]
By the Hahn--Banach theorem, $\partial f (x) \neq \emptyset$ for every $x$ in the interior of  $\dom(f)$.
The Fenchel--Young inequality states that $\langle x,y \rangle \le f(x) +f^\ast(y)$ for every $x,y \in \R^d$.
Moreover, $y \in \partial f(x)$ if and only if $(x,y)$ is an equality case in the Fenchel--Young inequality if and only if $x \in \partial f^\ast(y)$.
Further note that there is a Borel selector of $\partial f$ (that is a Borel function $g : \{ x \in \R^d : \partial f(x) \neq \emptyset\} \to \R^d$ such that $g(x) \in \partial f(x)$). This can be seen using the Arsenin--Kunugui selection theorem, see e.g.\ \cite[Theorem~18.8]{Ke95}.

Note that a convex function $f$ is Frechet-differentiable in $x \in \dom(f)$ if and only if it is Gateaux-differentiable in $x$ if and only if $\partial f(x)$ is a singleton, see e.g.\ \cite[Corollary~17.34]{BaCo11}. Moreover, if $f$ is differentiable on an open set $U \subset \dom(f)$, then $f \in C^1(U)$, see e.g.\ \cite[Corollary~17.34]{BaCo11}.

A convex function $f: \R^d \to \R \cup \{+\infty\}$ is called supercoercive if 
\begin{align}\label{eq:def:supercoerive}
    \lim_{|x| \to +\infty} \frac{f(x)}{|x|} = +\infty.
\end{align}

A convex function $f: \R^d \to \R \cup \{+\infty\}$ is supercoercive if and only if $f^\ast$ is finite, see e.g.\ \cite[Proposition~3.5.4]{BoVa10}. Given a finite supercoercive convex function $f$, we have that $f$ is strictly convex if and only if $f^\ast \in C^1(\R^d)$, see e.g.\ \cite[Corollary~18.12]{BaCo11}.

The infimal convolution of two proper functions $f, g : \R^d \to \R \cup \{+\infty\}$ is
\begin{align}\label{eq:def:infconv}
    (f\Box g)(x):=\inf_{y\in\R^d}\{f(x-y) + g(y)\}.
\end{align}
We then have 
\begin{align}\label{eq:BoxDualPlus}
  (f \Box g)^\ast = f^\ast + g^\ast \qquad \text{and} \qquad (f+g)^\ast = f^\ast \Box g^\ast,
\end{align}
where the second assertion is only true provided that $f,g$ are convex lsc and that intersection between $\dom(f)$ and the interior of $\dom(g)$ is not empty, see e.g.\ \cite[Proposition~13.21, Theorem~15.3]{BaCo11}.

Let $f : \R^d \to \R \cup \{+\infty \}$ be convex and lsc. Then $f \Box \frac{1}{2} | \cdot |^2$ is the Moreau-Yosida approximation of $f$. For $x \in \R^d$, the unique minimizer in the definition of $f \Box \frac{1}{2} | \cdot |^2(x)$ is called ${\rm Prox}_f x$, i.e.\
$
    \frac{1}{2}|\cdot|^2 \Box f(x) = \frac{1}{2}|x - {\rm Prox}_fx|^2 + f({\rm Prox}_fx).
$
The map $x \mapsto {\rm Prox}_f x$ is called proximal operator of $f$. Both ${\rm Prox}_f$ and $\id - {\rm Prox}_f$ are 1-Lipschitz, see e.g.\ \cite[Section 12.4]{BaCo11}.

If $f,g$ are proper convex lsc, and one of them is finite, differentiable and supercoercive, then $f \Box g \in C^1(\R^d)$, see e.g.\ \cite[Corollary~18.8]{BaCo11}. If $f: \R^d \to \R$ is continuous and $g : \R^d \to \R \cup \{ +\infty\}$ is convex, replacing it by its lsc envelope, does not affect their infimal convolution (see e.g.\ \cite[Exercise~12.10]{BaCo11}). Specifically, we have 
\begin{align}\label{eq:BoxLscHull}
    f \Box g = f \Box g^{\ast\ast}.
\end{align}

In the context of the growth condition \eqref{eq:WOTbound}, we often work with convex functions $h: [0,\infty) \to \R \cup \{+\infty\}$. These functions fit into the framework by extending them with the value $+\infty$ to a function on the whole real line. It is easy to see from \eqref{eq:def_legendre} that for such functions $h$, their convex conjugate $h^\ast$ is increasing as supremum of increasing functions. In the same spirit as the equivalence between supercoercivity and finiteness of the convex conjugate, we have the following: 
\begin{lemma}\label{lem:superlinear}
A lsc convex function $h: [0,\infty) \to [0,\infty)$ satisfies $\lim_{x \to +\infty} h^{\ast}(x)/x = +\infty$ if and only if $h(x) < \infty$ for all sufficiently large $x$. 
\end{lemma}
\begin{proof}
If $\lim_{x \to +\infty} h^{\ast}(x)/x < +\infty$, there are $\alpha,\beta >0$ such that $h^{\ast}(x) \le \chi_{[\alpha,\infty)}(x) + \beta x$. Taking the convex conjugation of this inequality yields, $h(x) = h^{\ast\ast}(x) \ge -\alpha (x-\beta) + \chi_{(-\infty,\beta]}(x)$, contradicting $h(x) < \infty$ for all sufficiently large $x$. 

Conversely, if $h(x)=+\infty$ for all $x \ge \gamma$, then $h(x) \ge \alpha + \beta x + \chi_{[\gamma,\infty)}$. By taking the convex conjugate, we have $h^\ast(x) \le -\alpha - \beta\gamma + \gamma x + \chi_{(-\infty,\beta]}(x)$ and hence $\lim_{x \to +\infty} h^{\ast}(x)/x \le \gamma < \infty$.
\end{proof}

\subsection{Support functions, convex sets and cones}\label{sec:app2}
A support function is a convex function $h : \R^d \to \R \cup \{+\infty \}$ that is positively homogeneous, i.e.\ $h(tx)=th(x)$ for every $x \in \R^d$ and $t \ge 0$. Note that support functions are subadditive, i.e., $h(x+y) \le h(x) + h(y)$.

There is a one-to-one correspondence between support functions and closed convex subsets of $\R^d$. The support function of a closed convex set $K \subset \R^d$ is given by
\begin{align}\label{eq:hK_app}
    h_K(x)=\sup_{y \in K} x \cdot y.
\end{align}
Note that we have 
$
    \partial h_K(0) =K.
$
As $h = h_{\partial h(0)}$ and $\partial h(0)$ is compact convex, every support function arises as in \eqref{eq:hK_app}. Moreover, $h_K$ is finite if and only if $K$ is compact. 

The convex indicator of a convex set $C \subset \R^d$ is given by $\chi_K(x)=0$ if $x \in K$ and $\chi_K(x)=+\infty$ if $x \notin K$. Note that $\chi_K$ is lsc if and only if $K$ is closed. In this case we have $\chi_K^\ast = h_K$ and $h_K^\ast = \chi_K$. Using \eqref{eq:BoxDualPlus} it follows that $h_{K \cap L} = h_K \Box h_L$ for $K,L \subset \R^d$ convex.

\begin{lemma}\label{lem:RangePartialBox}
Let $K \subset \R^d$ be a compact convex set and $\phi: \R^d \to \R \cup \{+\infty\}$ be a convex function. Then $\partial (h_K \Box \phi)(x) \subseteq K$ for every $x \in \R^d$. 
\end{lemma}
\begin{proof}
Let $x \in \R^d$ and $y \in \partial(h_K \Box \phi)(x)$. Then the subdifferential inequality yields for every $z \in \R^d$
\[
h_K \Box \phi(z) \ge h_K \Box  \psi (x) + \langle y, z - x \rangle.
\]
By considering both sides of the inequality as a function of $z$ and applying the convex conjugate, we find that for every $w \in \R^d$
\[
\chi_K(w) + \phi^\ast(w) \le - h_K \Box  \psi (x) + \langle x,y \rangle + \chi_{\{y\}}(w).
\]
As the right hand side of this inequality yields a finite value at $w=y$, the left-hand side does as well. Hence, $\chi_K(y) < \infty$, i.e.\ $y \in K$.
\end{proof}
The following characterization of convex functions that arise as infimal convolution with a given support function is straightforward to prove. 
\begin{lemma}\label{lem:RangeOfBox}
Let $\phi: \R^d \to \R \cup \{+\infty\}$ be convex and $h: \R^d \to \R \cup \{+\infty \}$ be a support function. Then the following are equivalent:
\begin{enumerate}[(i)]
    \item $\phi=h\Box \phi$
    \item $\phi = h \Box \psi$ for some convex function $\psi$
    \item $\phi(x_1) - \phi(x_2) \le h(x_1-x_2)$ for all $x_1,x_2 \in \R^d$ 
    \end{enumerate}
\end{lemma}
From this we immediately derive
\begin{corollary}\label{cor:RangeBoxIntersec}
Let $\phi: \R^d \to \R \cup \{+\infty\}$ be convex and $h_1,h_2: \R^d \to \R \cup \{+\infty \}$ be support functions and write $h=h_1 \Box h_2$. Then there is a convex function $\psi$ such that $\phi = \psi \Box h$ if and only if there are convex functions $\psi_1,\psi_2$ such that $\phi = h_1 \Box \psi_1$ and $\phi = h_2 \Box \psi_2$.
\end{corollary}

A cone is a convex set $C \subset \R^d$ such that for every $t \ge 0$ and $x \in C$, we have $tx \in C$. A cone is called pointed if it contains no non-trivial subspace. There is a one-to-one correspondence between pointed cones and partial orders on $\R^d$ that are compatible with the linear structure.

\bibliographystyle{abbrv}%plain
\bibliography{joint_biblio}

\end{document}